\RequirePackage{amsthm}
\documentclass[sn-mathphys]{sn-jnl}% Math and Physical Sciences Reference Style
\usepackage{graphicx}%
\usepackage{multirow}%
\usepackage{amsmath,amssymb,amsfonts}%
\usepackage{mathrsfs}%
\usepackage[title]{appendix}%
\usepackage{xcolor}%
\usepackage{textcomp}%
\usepackage{manyfoot}%
\usepackage{booktabs}%
\usepackage{algorithm}%
\usepackage{algorithmicx}%
\usepackage{algpseudocode}%
\usepackage{listings}
\usepackage{natbib}%
\usepackage{esvect}%
\usepackage[english]{babel}
\theoremstyle{thmstyleone}%
\newtheorem{theorem}{Theorem}%  meant for continuous numbers
\newtheorem{proposition}[theorem]{Proposition}% 

\theoremstyle{thmstyletwo}%
\newtheorem{remark}{Remark}%
\newtheorem{corollary}{Corollary}%
\newtheorem{assumptions}{Assumptions}%

\theoremstyle{thmstylethree}%

\begin{document}

\title[A cyclic random motion in $\mathbb{R}^3$ driven by geometric counting processes]{A cyclic random motion in $\mathbb{R}^3$ driven by geometric counting processes}


\author*[1]{\fnm{Antonella} \sur{Iuliano}}\email{antonella.iuliano@unibas.it}
\equalcont{These authors contributed equally to this work.}

\author[1]{\fnm{Gabriella} \sur{Verasani}}\email{gabriella.verasani@unibas.it}
\equalcont{These authors contributed equally to this work.}

\affil*[1]{\orgdiv{Department of Mathematics, Computer Science and Economics,}, \orgname{University of Basilicata}, \orgaddress{\street{Viale dell'Ateneo Lucano, 10}, \city{Potenza}, \postcode{85100}, \state{PZ}, \country{Italy}}}
%
%\affil[2]{\orgdiv{Department}, \orgname{Organization}, \orgaddress{\street{Street}, \city{City}, \postcode{10587}, \state{State}, \country{Country}}}
%
%\affil[3]{\orgdiv{Department}, \orgname{Organization}, \orgaddress{\street{Street}, \city{City}, \postcode{610101}, \state{State}, \country{Country}}}


\abstract{
We consider the random motion of a particle that moves with constant velocity in $\mathbb{R}^3$. 
The particle can move along four directions with different speeds that are attained cyclically. It follows that the support of the stochastic process describing the particle's position at time $t$ is a tetrahedron.
We assume that the sequence of sojourn times along each direction follows a geometric counting process (GCP). When the initial velocity is fixed, we obtain the explicit form of the probability law of the process $\boldsymbol{X}(t) = (X_1(t);X_2(t);X_3(t))$, $t > 0$, for the particle's position. 
We also investigate the limiting behavior of the related probability density when the intensities of the four GCPs tend to infinity.
Furthermore, we show that the process does not admit a stationary density.
Finally, we introduce the first-passage-time problem for the first component of $\boldsymbol{X}(t)$ through a constant positive boundary $\beta > 0$ providing the bases for future developments.
%
% The abstract serves both as a general introduction to the topic and as a brief, non-technical summary of the main results and their implications. Authors are advised to check the author instructions for the journal they are submitting to for word limits and if structural elements like subheadings, citations, or equations are permitted.
}



\pacs[ORCID]{Antonella Iuliano: Orcid 0000-0001-8541-8120 \\
	                        Gabriella Verasani: Orcid 0009-0000-4994-1694
}

\maketitle

\section{Introduction}\label{sec1}

In the last decades, finite-velocity random motions have been widely studied  as a natural class of stochastic processes to model real phenomena on multi-dimensional Euclidean spaces.  In the one-dimensional case, the model that describe these motions is the telegraph process, in which the changes of the two possible velocities are governed by the Poisson process. This process has been studied in \citep{orsingher1990probability}, \cite{kolesnik1998equations}, \cite{beghin2001probabilistic} and  \cite{kolesnik2023book}. Later, the study of random motions in two or more dimensions has been performed by many authors over the years. In particular,  finite-velocity planar random processes in continuous time have been investigated by \citep{orsingher1986planar} and \cite{masoliver1993some}.
Further investigations have been oriented to investigate the distribution for Markovian random motion in the plane (c.f. \cite{kolesnik1998equation} and  \cite{kolesnik2006discontinuous}), and  the random motion with possible reflections at Poisson-paced events (c.f. \cite{kolesnik2002analysis}). The analysis of these random motions with finite speed is performed analytically by solving partial differential equations. Other methods use probabilistic approaches based on order statistics, or on more general renewal processes as in \cite{di2002exact} (with arbitrary random steps between successive switches). 
\par
Several generalizations of the basic model have been proposed especially in biology and physics, motivated by the need of describing a variety of random motions performed by cells, micro-organisms and animals. These problems have been analyzed in \cite{martens2012probability}, \cite{hartmann2020convex}, \cite{santra2020run} and  \cite{pogorui2021distribution}.
\par 
The main aim for the study of finite-velocity random motions in one or more dimensions is the determination of the probability distribution of the position vector $\boldsymbol{X}(t)$ of a particle at time $t$. For instance, cyclic planar motions with three and four directions have been treated by \cite{orsingher2002bessel} and \cite{orsingher2004cyclic} in which the changes of direction are governed by a homogenous Poisson process. Moreover, the authors have found a connection between the equations governing the probability distributions and the Bessel functions of higher order. The probability law of the motion of a particle performing a cyclic random motion is determined in \cite{lachal2006cyclic} where the particle can take a finite number of  possible directions with different speeds. Here, the changes of direction occur at Poisson random times. For instance, see also \cite{lachal2006minimal} in which the probability distribution is obtained by using order statistics and is expressed in terms of hyper-Bessel functions of higher order. Recently, planar random motions with orthogonal directions switching at Poisson time have been examined in \cite{orsingher2019cyclic},  \cite{cinque2022random} and \cite{cinque2023stochastic}. Other remarkable results concerning random motions in high dimensional spaces with uniformly distributed directions are present in \cite{orsingher2007random}, \cite{pogorui2012evolution}, and \cite{pogorui2011isotropic}.
\par
Some modified versions of the telegraph process have been considered also by substituting the underlying Poisson process with different kinds of counting processes. For instance, in \cite{di2023some} the authors study a stochastic process
which describes the dynamics of a particle performing a finite-velocity random motion  in $\mathbb{R}$ and $\mathbb{R}^2$, whose velocities alternate cyclically along two and three different directions, respectively, with
possibly unequal velocities. The novelty of this last paper is that the number of 
displacements of the motion along each possible direction follows a Geometric Counting Process (GCP) (see \cite{cha2013note}). This new approach has the advantageous characteristic that consists in the possibility of describing 
phenomena whose interarrival times have heavy tails rather than the memoryless property, as observed often in real cases. 
Specifically, the properties of the interarrival times of  GCP's are declined in several areas, such as in software reliability or actuarial theory, where such counting processes are often considered to describe occurrences of shocks or claims, and in earth sciences (climatology, hydrology, etc) to model the failures along times. For instance, some examples of applied fields where GCPs can be used are discussed in Section 6 and 7 of  \cite{di2019some}.

\par 
Therefore, starting from the  results obtained in \cite{di2023some}, here, we propose an extension of  the classical telegraph process in $\mathbb{R}^3$ with the aim of  (i) determining  the explicit general laws of the distribution of the current position, and (ii) presenting an approach of resolution based on the study of the intertimes between consecutive changes of  speed. In particular, we assume that the motion proceeds along four velocities that alternate cyclically, where the intertimes between two subsequent velocity changes toward each direction are distributed as a GCP. The use of a limited number of speeds is justified by the results presented in \cite{orsingher2004cyclic} where the authors consider a cyclic random motion with four directions forming a regular tetrahedron in $\mathbb{R}^3$ but where the changes of speed are governed by a homogenous Poisson process.
Precisely, in our case the intertimes possess a heavy-tailed distribution, differently to the classical telegraph process in which the times between velocity changes have exponential distribution. The scheme of the present paper is useful in describing suitable dynamics where random occurrences between two events have infinite expectations, and are dependent. Thus, the proposed stochastic process provides new models for the description of phenomena that are no more governed by hyperbolic PDE's as the classical telegraph equation. Another benefit of the present study is the construction of new solvable models, whose probability laws are obtained in tractable and closed form (as in the special case with symmetry) using  an approach based on the analysis of the intertimes between consecutive velocity changes. Note that the introduction of these finite-velocity random motions is also useful to describe the movement of a particle that chooses the new direction among all the possible ones in a cyclical manner.

\par
Moreover, a further strength of this work concerns some interesting asymptotic results.  In particular, we are able to obtain the limiting density of the process in a closed form when the parameters of the intertimes between velocity changes tend to infinity. In addition, we show that the process does not admit a stationary density as $t$ goes to $+\infty$. All these motivations allow us to study the motion in  $\mathbb{R}^3$ because it can be applied to concrete situations and the amount of calculations needed is comparatively acceptable. For instance, we mention that these models can approximate the random motion of some physical models linked to the movement of particles in the environment, birds and animals in the ocean, flying objects or users in a shopping mall.
\par
In detail, we analyze the process $\{(\boldsymbol{X}(t),V(t)), t \geq 0 \}$ in  $\mathbb{R}^3$ which describes the position of randomly moving particle performing a cyclic alternating motion with four specific and possible different directions $\vec{v}_j$, for $1 \leq j \leq4$. The velocity-vectors form a (possibly irregular)  tetrahedron, say $\mathcal{T}(t)$, i.e., the set of all possible positions of the moving particle at time $t >0$ on the surface of the support  in $\mathbb{R}^3$. 
Note that the analysis of random motions in multidimensional Euclidean spaces is quite rare in the literature, since its analysis is rather difficult. Hence, to overcome the difficulties of the study we will refer in detail to the simpler case in which the region ${\cal T}(t)$ is regular, centered in the origin and with vertices placed on the coordinate axes.
Once defined the geometry of the region $\mathcal{T}(t )$, the probability law of the process is determined when (i) the initial components are given by three terms, describing the situations in which the particle is found on the vertices, edges and faces of $\mathcal{T}(t )$, at the beginning of its motion, and (ii) the density concerning the absolutely continuous part is related to the motion of the particle in the interior of $\mathcal{T}(t )$. In particular, for the latter absolutely continuous component we exhibit an integral representation involving the probability density functions and the conditional survival function of the intertimes  between two successive events.
\par
This is the plan of the paper. In Section \ref{sec2} we introduce the process $\{(\boldsymbol{X}(t),V(t)), t \geq 0\}$ describing the position of a particle performing a cyclic, minimal, random motion in $\mathbb{R}^3$ with four possible directions $\vec{v}_j$, for $1 \leq j \leq 4$, and constant velocity $c>0$. Some preliminary results on the GCP are also briefly illustrated with reference to the distribution of the intertimes between two consecutive events. Then, in Section \ref{sec4}, we study the directional vectors and the analytic representation of the support in $\mathbb{R}^3$ identifying the tetrahedron $\mathcal{T}(t)$. In Section \ref{sec5} we investigate the stochastic process and its probability laws with underlying GCP, and determine the explicit expression of the initial and absolutely continuous components.  Section \ref{sec6} illustrates a special case with four fixed cyclic directions forming a regular tetrahedron in $\mathbb{R}^3$. Moreover, we examine the limiting distribution of the process when the parameters of the intertimes tend to infinity and when the time $t$ goes to infinity. Finally, in Section \ref{sec7} we provide some basic lines for the study of  the first-passage-time problem of the first component of the process through a constant positive boundary $\beta >0$. 

\section{A random motion in $\mathbb{R}^3$ with cyclic velocities}\label{sec2}

Let $\boldsymbol{X}(t) = (X_1(t), X_2(t), X_3(t))$ and $V(t)$ be respectively the position and the velocity of the particle at an arbitrary time $t \geq 0$ in the space $\mathbb{R}^3$. Assume that each point in $\mathbb{R}^3$ is represented by a triple $\boldsymbol{x}=(x_1,x_2, x_3)$. We consider a cyclic random motion $\{(\boldsymbol{X}(t),V(t)), t \geq 0 \}$ performed by a particle which can take four possible directions $\vec{v}_j$, for $1 \leq j \leq4$, and moves with constant velocity $c > 0$. The motion is cyclic in the sense that at each event the particle switches from direction $\vec{v}_j$ to $\vec{v}_{j+1}$, for $j=1,2,3$, and then from $\vec{v}_4$ to $\vec{v}_1$, and so on. Let $D_{j,k}$ be the random duration of the $k$-th time interval during which the motion proceeds with velocity $c$.
For any $1 \leq j \leq4$, the set $D_{j,\cdot}:=\{D_{j,k}, k \in \mathbb{N}\}$ constitutes a sequence of non-negative and dependent absolutely continuous random variables such that
\begin{equation}
	D_j^{(0)}=0, \quad D_j^{(k)}=\sum_{i=1}^k D_{j,i}, \quad k \in \mathbb{N}.
	\label{eq1}
\end{equation}
Moreover, the sets $D_{j,\cdot}$, for $1 \leq j\leq 4$, are mutually independent.  %With reference to the intertimes  $D_{j,k}$, for all  ${x} \in \mathbb{R}$, we define the distribution function  $F_{D_{j,k}}({x}) = \mathbb{P}(D_{j,k} \leq {x}),$ the survival function  $\overline{F}_{D_{j,k}}({x})=1-F_{D_{j,k}}({x})$ and the corresponding probability density function (p.d.f.) $f_{D_{j,k}} ({x})$.
\par
Now, let $\{ N(t), t \geq 0 \}$ be the alternating counting process having arrival times $T_1 , T_2 , \ldots$ (i.e., the instants when the events occur), such that $N(t)$ counts the total number of velocity reversals of the particle in $[0, t ]$, i.e.
\begin{equation}
	N(0)=0, \quad N(t)=\sum_{k=0}^{\infty}{\bf 1}_{\{T_k \leq t\}}, \quad t \geq 0.
	\label{eq2}
\end{equation}
In order to obtain the probability law of the stochastic process $\{(\boldsymbol{X}(t),V(t)), t \geq 0\}$ introduced so far,  we consider the stochastic equations for the position $\boldsymbol{X}(t)$ and the velocity $V(t)$ of the particle at time $t\geq 0$. Specifically, we have
\begin{equation*}
	V(t)=\sum_{k=0}^{\infty}{\bf 1}_{\{T_{k-1} \leq t < T_k\}} \vec{v}_k,
	%\label{eq3}
\end{equation*}
with $V (0) \in \{\vec{v}_1,\vec{v}_2,\vec{v}_3,\vec{v}_4\}$, and
\begin{equation}
	\boldsymbol{X}(t)=\boldsymbol{X}(0)+\int_{0}^{t} V(s) ds=\sum_{k=1}^{N(t)-1}\big(T_k-T_{k-1}\big)\vec{v}_k+\big(t-T_{N(t)-1}\big)\vec{v}_{N(t)}, 
	\label{eq4}
\end{equation}
with $\boldsymbol{X}(0)=\boldsymbol{0} \equiv (0,0,0)$ and $\sum_{k=1}^{N(t)-1}(T_k-T_{k-1})=T_{N(t)-1}$, since $T_0=0$. The sum in Eq.\,(\ref{eq4}) refers to the case where at least one change of direction has occurred before $t$, while the last term is related to the displacement along the current direction at time $t$. 
\par
We indicate with $T_{k}$ the $k$-th random instant in which the motion modifies its direction, for $n \in \mathbb{N}_{0}$. Remembering Eq.\,(\ref{eq1}) the following identity holds:
\begin{equation} 
	T_{4n+j}=\sum_{r=1}^{4}D_r^{(n+m_{j,r})}, 
	%T_{4n+i+j+k+1}=D_{1}^{(n+1)}+D_{2}^{(n+i)}+D_{3}^{(n+j)}+D_{4}^{(n+k)}, \ i,j,k=0,1,  \ \geq j\geq k.
	\label{eq5}
\end{equation}
where, for fixed $j \in \{1,2,3,4\}$, $M=\big(m_{j,r}\big)_{1\leq r \leq 4} \in \mathbb{R}^{4 \times 4}$ is equal to
$$
M=\begin{pmatrix}
	1 & 0 & 0 &0\\
	1 & 1 & 0 &0\\
	1 & 1 & 1 &0\\
	1 & 1 & 1 &1
\end{pmatrix}.
$$

The relations obtained above will be used in Section \ref{sec5} to study the probability law of the process $\{(\boldsymbol{X}(t),V(t)), t \geq 0 \}$.

\subsection{The distribution of the intertimes as a Geometric Counting Process}\label{sec3}

We introduce some preliminary results on the GCP (for more details, see \cite{cha2013note} and \cite{di2019some}).  Specifically, we assume that the random intertimes between consecutive changes of directions are governed by possible different GCPs (see, for instance, \cite{di2023some}). 
% Differently from the Poisson process, the GCP does not possess the property of independent increments, which usually complicates the probabilistic analysis. Moreover, the intertimes of the GCP have modified Pareto distribution, thus such process is suitable to be applied in those area where random occurrences between events have infinite expectations, and are not independent.
\par
For this purpose, we consider a mixed Poisson process $\{N(t), t \geq 0\}$ characterized by the following marginal distribution expressed as a mixture:
\begin{equation*}
	\mathbb{P}[N(t)=k]= \int_{0}^{t} \mathbb{P}[N^{(\alpha)}(t)=k]dU(\alpha),\quad  t \geq 0, \quad k \in \mathbb{N}_{0},
	%\label{eq10}
\end{equation*}
where $\{N^{(\alpha)}(t)$, $ t \geq 0\}$  is a Poisson process with intensity $\alpha$ and $U$ is an exponential distribution with support $\mathbb{R}^+$ Here, we consider the special case when $U(\cdot) = U_{\lambda}(\cdot)$ is an exponential distribution with mean $\lambda \in \mathbb{R}^{+}$. Therefore, we refer to the  process $N(t)$ as a GCP with intensity $\lambda$, according to \cite{cha2013note}, where the authors studied dependence properties of its increments in the general case of non-constant intensities. The probability distribution of  the process $N(t)$ satisfies the following properties:
\begin{itemize}
	\item at time $t = 0$, one has  $N(0) = 0$
	\item for all $s,t \geq 0$ and $k \in \mathbb{N}_{0}$ it holds:
	\begin{equation*}
		\mathbb{P}\{N(t+s)-N(t)=k\}= \dfrac{1}{1+\lambda s} \bigg( \dfrac{\lambda s}{1+\lambda s}\bigg)^{k}.
		%\label{eq11}
	\end{equation*}
\end{itemize}
Let $T_{k}$, with  $k\in \mathbb{N}$, be the random times denoting the arrival instants of the process $N(t)$ such that $T_{0}=0$. It is interesting to observe that the probability density function (p.d.f.) of $T_k$ is expressed as follows:
\begin{equation*}
	f_{T_k}(t)=k\bigg( \dfrac{\lambda t}{1+\lambda t}\bigg)^{k-1}\dfrac{\lambda}{(1+\lambda t)^{2}}, \quad t \geq 0.
	%\label{eq12}
\end{equation*}
We define $D_{j,k}=T_{k}-T_{k-1}$, with $1 \leq j \leq 4$ and $k \in \mathbb{N}$, the increments between two consecutive events. Differently from the Poisson process, the GCP does not have the property of independent increments, hence, the random variables $D_{j,k}$ are dependent. 
Moreover, making use of Eq.\,(10) in \cite{di2019some}, we get the marginal density for all intertimes $D_{j,k}$, $1 \leq j \leq 4$ and $k \in \mathbb{N}$, 
\begin{equation}
	f_{D_{j,k}}(t)=\frac{\lambda}{(1+\lambda\, t)^{2}}, \quad t \geq 0.
	\label{eq13}
\end{equation}
We remark that $D_{j,k}$ has  a long right tail (see, for instance, \cite{Asmussen2003}) since, for all $t > 0$,
$$
\lim_{x \to\infty} \mathbb{P}\{D_{j,k}>x+t | D_{j,k}>x\} = 1.
$$
From density (\ref{eq13}) it is not hard to see that $D_{j,k}$ has a modified Pareto distribution which means that the intertimes of the GCP $\{N(t), t >0\}$ have non-finite expectations. In addition, to obtain the explicit expressions of the process $\boldsymbol{X}(t)$, we recall the conditional survival function of $D_{j,k}$ conditional on $T_{k-1}=t$, for $k \in \mathbb{N}$, is given by
\begin{equation}
	\overline{F}_{D_{j,k}\vert T_{k-1}}(s \vert t)=\bigg(\dfrac{1+\lambda t}{1+\lambda(t+s)}\bigg)^{k}, \quad s,t \geq 0, 
	\label{eq14}
\end{equation}
and the corresponding probability density function (p.d.f.) of $D_{j,k}$ conditional on $T_{k-1}=t$, for $k \in \mathbb{N}$, i.e., 
\begin{equation}
	f_{D_{j,k}\vert T_{k-1}}(s \vert t)=\frac{k\lambda (1+\lambda t)^{k}}{[1+\lambda(t+s)]^{k+1}}, \quad s,t \geq 0.
	\label{eq15}
\end{equation}
%
% For more details about Eqs.\,(\ref{eq14}) and (\ref{eq15}), see  Section 2 of \cite{di2019some}. 

\par
In the following, we determine the initial and absolutely continuous components of the probability law of the stochastic process $\{(\boldsymbol{X}(t),V(t)), t \geq 0\}$, which describes a cyclic alternating random motion along four directions $\vec{v}_j$, for $1 \leq j \leq 4$, driven by independent GCPs with intensity ?$\lambda_j$ ($1 \leq j \leq 4$.
We remark that the intertimes of the GCPs are in practice more realistic since in real cases the counting processes under investigation do not have necessarily the independent increments property. Some applications may be found in geophysics (see \cite{Benson2007}), in climatology (see \cite{Levergnat1998}) or in modeling for internet traffic (see \cite{clegg2010}).  Moreover, in real applications the number of events in a fixed time interval does not always follow the Poisson distribution. Therefore, the class of GCPs is a possible alternative to Poisson processes. At least, the use of GCPs allows us to obtain the probability law of the process $\boldsymbol{X}(t)$ in a tractable way as shown in the remainder of the paper.

\section{The velocity vectors and the support}\label{sec4}

In this section, we discuss the properties of the directional velocity vectors  and the geometrical features of  the region generated by them, i.e., the set of all possible positions of the moving particle at time $t>0$. 

\subsection{The directional vectors}\label{subsec4_1}

Let $\vec{v}_j$, $1 \leq j \leq 4$, be the vectors representing the possible directions of the cyclic motion in $\mathbb{R}^3$ (i.e., the minimal number of directions for a non trivial motion in $\mathbb{R}^3$). The velocity direction coordinates can be defined by using the spherical coordinates, given by
\begin{equation*} 
	\vec{v}_{j}=(\vec{l} \cos\theta_{j} \sin\varphi_{j}  + \vec{m} \sin\theta_{j} \sin\varphi_{j} + \vec{n}\cos\varphi_{j}),
	%\label{eq16}
\end{equation*}
where $\vec{l}, \vec{m}, \vec{n}$ are the unit vectors along the Cartesian coordinate axes in $\mathbb{R}^3$, $\varphi_j$ is the azimuthal angle in the $xy$-plane from the $x$-axis with $\theta_j \in [0,2\pi]$ and $\theta_{j}$ is the polar angle from the positive $z$-axis with $\varphi_j \in [0,\pi]$, for $1 \leq j \leq 4$.
\par
The particle starts from the origin $\boldsymbol{0}=(0,0,0)\in \mathbb{R}^3$ at time $t = 0$, running with constant velocity $c > 0$. 
% The particle starts at the origin $\boldsymbol{0}=(0,0,0)$ of the space $\mathbb{R}^3$ at time $t=0$, i.e $\boldsymbol{X}(0)=\boldsymbol{0}$, moving with a constant velocity $v_j>0$, for $1 \leq j \leq 4$. 
%With reference to Eq.\,(\ref{eq4}), the times between two switches, i.e.\ $T_{k} - T_{k-1}$, for $k \geq  1$, are dependent random variables and represent the intertimes of  four GCP's with possible different intensities. 
% Specifically, the times between two changes is given by $T_{4k+j} - T_{4k+j-1}$, for $k\geq 0$, i.e., the $(k+1)$-th intertime of the $j$-th GCP, having intensity $\lambda_j$, for $1 \leq j \leq 4$.  
%
Initially, it moves along the direction $\vec{v}_{1}$ . Then, after a random duration $D_{1,1}$, the particle switches instantaneously its speed, moving along  $\vec{v}_{2}$  for a random duration  $D_{2,1}$. Subsequently, it goes along $\vec{v}_{3}$ and $\vec{v}_{4}$ for a length of time $D_{3,1}$ and $D_{4,1}$, respectively. Thus, the particle motion proceeds cyclically with velocities $\vec{v}_j$ for the random periods $D_{j,2}, D_{j,3}, D_{j,4}, \ldots$, such that,  for each $1 \leq j \leq 4$ and $k \in \mathbb{N}$, we have
\begin{equation}
	D_{j,4k+i}\overset{d}{=}D_{j,i},
	\label{eq17}
\end{equation}
where $\overset{d}{=}$ means equality in distribution. Hence, during the $n$-th cycle the particle moves along directions $\vec{v}_j$ in sequence for the random lengths $D_{j,n}$,  with $1 \leq j \leq 4$ and $n \in \mathbb{N}$. In other words, the particle runs along direction $\vec{v}_j$ and, after an intertime distributed as a CGP, it takes the direction $\vec{v}_{j+1}$, with $\vec{v}_{j+4n}=\vec{v}_j$, for $1 \leq j \leq 4$ and $n \in \mathbb{N}$. We note that the results concerning the cases of other initial direction, i.e., $\vec{v}_2$, $\vec{v}_3$, $\vec{v}_4$, may be easily deduced.
Now, in order to define the set of all possible positions occupied by the particle at an arbitrary time $t >0$, we consider the following  Remark.
\begin{remark} \label{prop1}
	Let be $\{(\boldsymbol{X}(t),V(t)), t \geq 0\}$ the stochastic process defined in Section \ref{sec2}. Given the direction vectors $\vec{v}_j$ with $\vec{v}_{j+4n}=\vec{v}_j$, for $1 \leq j \leq 4$ and $n \in \mathbb{N}$, we have that the particle motion reaches any state of $\mathbb{R}^3$ in a sufficiently large time $t>0$, if and only if, %for a fixed $j$,
	\begin{enumerate}
		\small
		\item[(i)] the set of direction vectors $\{\vec{v}_1, \vec{v}_2, \vec{v}_3\}$ are linearly independent;
		\item[(ii)] the direction vector $\vec{v}_4\in S$ with 
		\begin{equation*} 
			S=\Big\{\xi(-\vec{v}_1)+\eta(-\vec{v}_2)+(1-\xi-\eta) (-\vec{v}_3),  0 \ \leq \xi,\eta \leq 1, \ \xi+\eta \leq 1\Big\}.
			%\label{eq18}
		\end{equation*}
	\end{enumerate}
\end{remark}
Hence, the set of all possible positions under the hypothesis of Remark \ref{prop1} identifies as state space at time $t$ a tetrahedron whose vertices coincide with the endpoints of the vectors $ct\vec{v}_j$.  In other words, the tetrahedron is  defined as a $3$-dimensional simplex where the interior points are convex combinations of the four vertices such that $\sum_{j=1}^{4} \alpha_j \vec{v}_j$ with $\sum_{j=1}^{4}\alpha_j =1$ and $\alpha_j \geq 0$. 

\subsection{Analytic representation of  the support}\label{subsec4_2}

Let us consider the points $A_j(t)$, $1 \leq j \leq 4$, defined by $\vv{\boldsymbol{0}A_j(t)}=\vec{v}_j t$, where $\boldsymbol{0}=(0,0,0)$ is the origin of the Cartesian coordinate system. The points $A_j(t)$, $1 \leq j \leq 4$, given by
\begin{equation}
	A_{j}(t)=(ct\cos\theta_{j}\sin\varphi_{j}, ct\sin\theta_{j} \sin\varphi_{j}, ct\cos\varphi_{j}), \quad 1 \leq j \leq 4
	\label{eq19}
\end{equation}
are the vertices of  the  support in $\mathbb{R}^3$ at time $t>0$, i.e., a time-dependent tetrahedron with edges $E_{ij}(t)$ and faces $F_{ijk}(t)$, $1 \leq i,j,k\leq 4$ and $i<j<k$. 
In particular, if the three different vertices $A_i(t)$, $A_j(t)$ and $A_k(t)$ are not aligned, then the particle reaches any point of the support in $\mathbb{R}^3$. This is true since the matrix generated by two adjacent segments $\overrightarrow{A_i (t)A_k(t)}$ and $\overrightarrow{A_k(t) A_j(t)}$ has maximum rank, i.e. 2. Therefore, under the assumptions of Remark \ref{prop1} the equations linking two adjacent vertices can expressed as
\begin{equation}
	a_{ijk}x_1+b_{ijk}x_2 + c_{ijk}x_3 -ctq_{ijk}=0, \quad1 \leq i,j,k \leq 4, \quad i<j<k,
	\label{eq20}
\end{equation}
where
\begin{equation*}
	\begin{aligned}
		a_{ijk}&=(\sin\theta_{j}\sin\varphi_{j}-\sin\theta_{i}\sin\varphi_{i})(\cos\varphi_{k}-\cos\varphi_{i})\\ &-(\sin\varphi_{k}\sin\theta_{k}-\sin\varphi_{i}\sin\theta_{i})(\cos\varphi_{j}- \cos\varphi_{i}), 
	\end{aligned}
	%\label{eq21}
\end{equation*}
\begin{equation*}
	\begin{aligned}
		b_{ijk}&=(\sin\varphi_{k}\cos\theta_{k}-\sin\varphi_{i}\cos\theta_{i})(\cos\varphi_{j}-\cos\varphi_{i})\\
		&-(\sin\varphi_{j}\cos\theta_{j}-\sin\varphi_{i}\cos\theta_{i})(\cos\varphi_{k}- \cos\varphi_{i}),
	\end{aligned}
	%\label{eq22}
\end{equation*}
\begin{equation*}
	\begin{aligned}
		c_{ijk}&=\sin\varphi_{i}\cos\theta_{i}(\sin\varphi_{j}\sin\theta_{j}-\sin\varphi_{k}\sin\theta_{k})\\
		&+\sin\varphi_{j}\cos\theta_{j}(\sin\varphi_{k}\sin\theta_{k}-\sin\varphi_{i}\sin\theta_{i})\\
		&+ \sin\varphi_{k}\cos\theta_{k}(\sin\varphi_{i}\sin\theta_{i}-\sin\varphi_{j}\sin\theta_{j}),
	\end{aligned}
	%\label{eq23}
\end{equation*}
and, 
\begin{equation*}
	q_{ijk}=a_{ijk}\cos\theta_{i}\sin\varphi_{i}+ b_{ijk} \sin\theta_{i}\sin\varphi_{i}+c_{ijk} \cos\varphi_{i}.
	\label{eq24}
\end{equation*}
\par
We recall that $V(t)$ is the current direction of motion at time $t$. Assuming that $V(0)$ takes value $\vec{v}_1$ with velocity $c >0$, we are able to determine the analytic expression of the support in $\mathbb{R}^{3}$,  say $\mathcal{T}(t)$, i.e., the set of all positions allocated by the particle when it is confined in a tetrahedron, defined as
\begin{equation}
	\begin{aligned} 
		\mathcal{T}(t)=\left\{ \begin{array}{rl}
			(x_1,x_2, x_3)\in \mathbb{R}^{3}:  
			\begin{cases}
				a_{123}x_1+b_{123}x_2+c_{123}x_3 -ctq_{123} \geq 0\\
				a_{124}x_1+b_{124}x_2+c_{124}x_3 -ctq_{124} \geq 0\\
				a_{234}x_1+b_{234}x_2+c_{234}x_3 -ctq_{234} \geq 0\\
				a_{134}x_1+b_{134}x_2+c_{134}x_3 -ctq_{134} \leq 0
			\end{cases}
		\end{array}
		\right\},
	\end{aligned}
	\label{eq25}
\end{equation}
where the given conditions derive from Eqs.\,(\ref{eq20}). In general, the particle motion includes different mutually exclusive cases based on the initial assumptions $\boldsymbol{X}(0)=\boldsymbol{0}$ and $V(0)=\vec{v}_j$, for a fixed $1\leq j\leq 4$. More precisely, the particle is situated on the vertices $A_j(t)$ of  $\mathcal{T}(t)$ if no velocity change occurs up to time $t$, while if one event occurs then, at time $t$, it will be located on some edge $E_{ij}$ of $\mathcal{T}(t)$.
Two events allow the particle to reach one of the faces $F_{ijk}$ of  $\mathcal{T}(t)$ and three or more changes of direction force the particle to be located inside $\mathcal{T}(t)$. Therefore, recalling Remark \ref{prop1}, the following holds.
\begin{remark}\label{remark:new}
For the probability law of the process $\{(\boldsymbol{X}(t),V(t)), t \geq 0\}$ we have:
		\begin{enumerate}
			\item[(i)] the first component is related to the case when no event occurs in the interval $(0,t)$, i.e., $N(t)=0$, and the particle is concentrated on the vertices of $\mathcal{T}(t)$; 
			\item[(ii)] the second component is concerning the situation in which one event happens at time $t$, i.e., $N(t)=1$, and the particle is located on some edge of $\mathcal{T}(t)$;
			\item[(iii)] the third component corresponds to the instance when two events occur,  i.e., $N(t)=2$ and the particle reaches one of the faces of $\mathcal{T}(t)$;
			\item[(iv)] the forth component, which is absolutely continuous, refers to the case when $N(t) \geq 3$, i.e. when the particle is placed strictly in the interior of $\mathcal{T}(t)$, namely ${\rm Int}(\mathcal{T}(t))$.
		\end{enumerate}
\end{remark}

\section{The probability law}\label{sec5}

To define the conditional distributions of the process $\{(\boldsymbol{X}(t),V(t)), t \geq 0\}$, we denote by 
	\begin{equation}
		C_i:=\{\boldsymbol{X}(0)=\boldsymbol{0},V(0) = \vec{v}_i\},
		\label{eq6}
	\end{equation}
	the event that the particle starts its motion from the origin $\boldsymbol{0}$ when the initial velocity is $\vec{v}_i$, $1\leq i \leq 4$. Therefore, let $\mathbb{P}_i$ be the probability conditional on $C_i$.

\par
The conditional probability laws given $C_i$ are composed by two contributions:
\begin{enumerate}
		\item[(i)] the initial components given by the terms related to the beginning of the motion. Specifically, they describe the cases in which the particle is over the border of the support $\mathcal{T}(t)$ in $\mathbb{R}^3$.
		\item[(ii)] The absolutely continuous component related to the particle motion inside the support $\mathcal{T}(t)$ in $\mathbb{R}^3$.
\end{enumerate}
For the case (ii) we define
	\begin{equation} 
		\begin{aligned}
			p _{ij} ( \boldsymbol{x}, t ) d \boldsymbol{x}&= \mathbb{P}_i\{ \boldsymbol{X}( t) \in d \boldsymbol{x} , V(t) =\vec{v}_j \}\\
			&=\mathbb{P}_i \{X_1(t) \in d x_1 , X_2( t) \in d x_2, X_3( t) \in d x_3, V(t) =\vec{v}_j\},
		\end{aligned}
		\label{eq7}
	\end{equation}
	with $1 \leq i,j \leq4$ and where $d \boldsymbol{x}$ is the infinitesimal element in the space $\mathbb{R}^3$ with the Lebesgue measure  $\mu(d  \boldsymbol{x})=dx_1 dx_2 dx_3$ ($i$ refers to the initial velocity, and $j$ to the  current velocity at time $t$). 
	The right-hand-side of Eq.\,(\ref{eq7}) represents the probability that the particle at time $t>0$ is located in a neighborhood of $\boldsymbol{x}\in \mathbb{R}^3$  and moves along direction $\vec{v}_j$, given the initial condition represented by $C_i$. Moreover, we can introduce the probability density function (p.d.f.) of the particle location $\boldsymbol{X}$(t) at time $t>0$, i.e.,
	\begin{equation} 
		p_{i}(\boldsymbol{x},t)= \mathbb{P}_i\{\boldsymbol{X}(t)\in d\boldsymbol{x}\}, \quad  1 \leq i \leq 4.
		\label{eq8}
	\end{equation}
	Due to Eqs.\,(\ref{eq7}) and (\ref{eq8}), one immediately has
	\begin{equation} 
		p_i(\boldsymbol{x},t)= \sum_{j=1}^{4}p_{ij}(\boldsymbol{x},t), \quad \quad  1 \leq i \leq 4.
		\label{eq9}
	\end{equation}   

\par
In order to determine the probability law of the stochastic process $\{(\boldsymbol{X}(t),V(t)), t \geq 0\}$, we assume that the starting direction at time $t$ is $\vec{v}_1$, so that, the initial condition is $C_1$ as defined in Eq.\,(\ref{eq6}). Therefore, the following results are obtained by taking into account the cases (i)-(iv) considered in Remark \ref{remark:new}.
\begin{theorem}\label{theor1}
	(Initial components) Let be $\{(\boldsymbol{X}(t),V(t)), t \geq 0\}$ the stochastic process defined in Section \ref{sec2}.  For $t \geq 0$ we have
	\begin{equation}
		\begin{aligned}
			\mathbb{P}_1\{\boldsymbol{X}(t)&=(ct\cos\theta_{j}\sin\varphi_{j}, ct \sin \theta_{j}\sin\varphi_{j},ct\cos\varphi_{j}), V(t)=\vec{v}_{1}\}\\
			&=\mathbb{P}\{D_{1,1} >t\},
			%\Bar{F}_{D_{1,1}}(t),
		\end{aligned}
		\label{eq26}
	\end{equation}
	\begin{equation}
		\mathbb{P}_1\{\boldsymbol{X}(t)\in E_{12}(t), V(t)=\vec{v}_{2}\}=\mathbb{P}\{D_{1,1} < t \leq D_{1,1}+D_{2,1}\}, %F_{D_{1,1}+D_{2,1}}(t)-F_{D_{1,1}}(t)
		\label{eq27}
	\end{equation}
	and
	\begin{equation}
		\mathbb{P}_1\{\boldsymbol{X}(t)\in F_{123}(t), V(t)=\vec{v}_{3}\}
		= \mathbb{P}\{D_{1,1}+D_{2,1} < t \leq D_{1,1}+D_{2,1}+D_{3,1}\},
		%F_{D_{1,1}+D_{2,1}+D_{3,1}}(t)-F_{D_{1,1}+D_{2,1}}(t).
		\label{eq28}
	\end{equation}
	where $D_{j, k}$ is the random duration of the $k$-th time interval during which the particle moves with velocity $\vec{v}_j$, for $1 \leq j \leq 3$ and $k \in \mathbb{N}$.
\end{theorem}
\begin{proof}
	Eqs.\,(\ref{eq26}), (\ref{eq27}) and (\ref{eq28}), when  $V(0)=\vec{v}_1$,
	are easily obtained, under the conditions (i), (ii) and (iii) of  Remark $\ref{remark:new}$.
\end{proof}
Using Eq.\,(\ref{eq7}) and recalling (iv) of  Remark $\ref{remark:new}$, we determine the absolutely continuous components of the probability law of the $X(t)$ conditioned on $C_1$, for $t>0$ and $1 \leq j \leq 4$, i.e., the densities
\begin{equation}
	p_{1j} (\boldsymbol{x}, t) d \boldsymbol{x} = \mathbb{P}_1 \{ \boldsymbol{X}(t) \in d \boldsymbol{x} , V ( t ) = \vec{v}_j\}.
	\label{eq29}
\end{equation}
In order to give the expression of  $p_{1 j} ( \boldsymbol{x}, t)$ we introduce the linear map $\zeta: \mathbb{R}^{4} \longrightarrow \mathbb{R}^{4}$ defined by:
\begin{equation} 
	\zeta(\boldsymbol{t})= \bigg(c\sum_{j=1}^{4} x_{1_{\vec{v}_{j}}}t_{j},c\sum_{j=1}^{4} x_{2_{\vec{v}_{j}}} t_{j},c\sum_{j=1}^{4} x_{3_{\vec{v}_{j}}}t_{j},\sum_{j=1}^{4}t_{j}\bigg),
	\label{eq30}
\end{equation}
for $\boldsymbol{t}=(t_1,t_2,t_3,t_4)$, and where $x_1{_{\vec{v}_{j}}}$, $x_2{_{\vec{v}_{j}}}$ and $x_{3_{\Vec{v}_{j}}}$ are the components of the vectors $\vec{v}_{j}$ ($1\leq j \leq 4$) respectively along  the $\boldsymbol{x}$-axes. When a cycle of the random motion ends after a period $\boldsymbol{t}$ the function $\zeta(\boldsymbol{t})$, given in Eq.\,(\ref{eq30}), provides a vector containing the displacements performed along the $\boldsymbol{x}$-axes during the cycle, as well as its whole duration. 
\par  
Moreover, we consider the transformation matrix $\boldsymbol{A}$ of function $\zeta(\boldsymbol{t})$ expressed in terms of spherical coordinates, i.e.,
\begin{equation} 
	\boldsymbol{A}=  \begin{pmatrix}
		c \cos\theta_{1} \sin\varphi_{1}&  c\cos\theta_{2} \sin\varphi_{2}&  c\cos\theta_{3} \sin\varphi_{3}&  c \cos\theta_{4}\sin\varphi_{4}\\
		c \sin\theta_{1} \sin\varphi_{1}&  c\sin\theta_{2} \sin\varphi_{2} &  
		c \sin\theta_{3} \sin\varphi_{3}&  c \sin\theta_{4}\sin\varphi_{4}\\
		c\cos\varphi_{1} & c\cos\varphi_{2} & c\cos\varphi_{3} & c\cos\varphi_{4},\\
		1 & 1 & 1 & 1
	\end{pmatrix}
	\label{eq31}
\end{equation}
with
\begin{equation}
	\begin{aligned}
		det \boldsymbol{A}&=c^{3}\Big\{ \sin\varphi_{1}\sin\varphi_{2}[\sin(\theta_{2}-\theta_{1})](\cos\varphi_{3}-\cos\varphi_{4})\\
		&+\sin\varphi_{1}\sin\varphi_{3}[\sin(\theta_{3}-\theta_{1})](\cos\varphi_{4}-\cos\varphi_{2})\\
		&+ \sin\varphi_{1}\sin\varphi_{4}[\sin(\theta_{4}-\theta_{1})](\cos\varphi_{2}-\cos\varphi_{3})\\
		&+\sin\varphi_{2}\sin\varphi_{3}[\sin(\theta_{3}-\theta_{2})](\cos\varphi_{1}-\cos\varphi_{4})\\
		&+ \sin\varphi_{2}\sin\varphi_{4}[\sin(\theta_{4}-\theta_{2})](\cos\varphi_{3}-\cos\varphi_{1})\\
		&+\sin\varphi_{3}\sin\varphi_{4}[\sin(\theta_{4}-\theta_{3})](\cos\varphi_{1}-\cos\varphi_{2})	\Big\}.
	\end{aligned}
	\label{eq32}
\end{equation}
\begin{remark}\label{remark1}
	Recalling (\ref{eq19}), we observe that the determinant $\det(t\boldsymbol{A})$ represents the volume $Vol(\mathcal{T}(t))$ of $\mathcal{T}(t)$ defined in (\ref{eq25}). 
\end{remark}
Now, given a sample path of the process $\{(\boldsymbol{X}(t),V(t)), t \geq 0\}$, we denote by $\tau_j = \tau_j(\boldsymbol{x})$ the non negative random variables representing the dwelling times of the particle motion in each direction $\vec{v}_{j}$ ($1\leq j \leq 4$) during $[0,t]$ (i.e., the residence times for the process $\boldsymbol{X}(t)$ on the event that $\boldsymbol{X}(t) = \boldsymbol{x}$, $V(t)=\vec{v}_j$), so that
\begin{equation*} 
	\sum_{j=1}^{4} \tau_{j}=t.
	%\label{eq33}
\end{equation*}
Therefore, recalling (\ref{eq31}) and (\ref{eq19}), we can express $\tau_j$, $1\leq j \leq 4$, as the 4-uple
\begin{equation*} 
	\zeta^{-1}(\boldsymbol{x}, t)=(\tau_{1}(\boldsymbol{x}, t), \tau_{2}(\boldsymbol{x}, t), \tau_{3}(\boldsymbol{x}, t),\tau_{4}(\boldsymbol{x}, t)),
	%\label{eq34}
\end{equation*}
with
\begin{equation} 
	\tau_j(\boldsymbol{x},t)=\sum_{k=1}^{4}\sigma_{jk}x_k+\sigma_{j4}t,
	\label{eq35}
\end{equation}
where $\sigma_{ij}=\frac{1}{det \boldsymbol{A}} \, cof_{ij} (\boldsymbol{A})$ with  $cof_{ij} (\boldsymbol{A})=(-1)^{i+j} det A_{ij}$, $1 \leq i,j \leq 4$.
Moreover, since the function $\zeta(\boldsymbol{t})$, for all $\boldsymbol{x}\in \mathbb{R}^{3}$, is bijective (see, for more details, \cite{lachal2006cyclic}), it is easy to see that $\boldsymbol{\tau}=(\tau_1,\tau_2,\tau_3,\tau_4)^T$ is solution of the system $\boldsymbol{A \tau} = (\boldsymbol{x},t)^T$. More precisely, in the following proposition the explicit form of Eq.\,(\ref{eq35}) can be obtained.
\begin{proposition}\label{prop2}
	Making use of Eq.\,(\ref{eq35}), for all $\boldsymbol{x} \in \mathbb{R}^3$ the coordinates $\tau_j$, for $1\leq j \leq 4$ and $c>0$,  are expressed as 
	\begin{equation} 
		\tau_{j}= \dfrac{c^{2}}{det\boldsymbol{A}} \bigg[L_{j}x_1 + M_{j}x_2+ N_{j}x_3 -ctP_{j}\bigg], \quad 1\leq j\leq 4,
		\label{eq36}
	\end{equation}
	where 
	\begin{equation*}
		\begin{aligned}
			L_{j}&=\sum_{i=1}^{3}	\sin\varphi_{j\hat{+}i}\sin\theta_{j\hat{+}i}\bigg(\cos\varphi_{j\hat{+}(i+1)}-\cos\varphi_{j\hat{+}(i+2)}\bigg),\\
			M_{j}&=\sum_{i=1}^{3}\sin\varphi_{j\hat{+}i}\cos\theta_{j\hat{+}i}\bigg(\cos\varphi_{j\hat{+}(i+2)}-\cos\varphi_{j\hat{+}(i+1)}\bigg),\\
			N_{j}&=\sum_{i=1}^{3}\sin\varphi_{j\hat{+}i}\cos\theta_{j\hat{+}i}\bigg(\sin\varphi_{j\hat{+}(i+1)}\sin\theta_{j\hat{+}(i+1)}-\sin\varphi_{j\hat{+}(i+2)}\sin\theta_{j\hat{+}(i+2)}\bigg),\\
			P_{j}&=\sum_{i=1}^{3}\sin\varphi_{j\hat{+}i}\cos\theta_{j\hat{+}i}
			\\
			&\times  \bigg(\sin\varphi_{j\hat{+}(i+1)}\sin\theta_{j\hat{+}(i+1)}\cos\varphi_{j\hat{+}(i+2)}-\sin\varphi_{j\hat{+}(i+2)}\sin\theta_{j\hat{+}(i+2)}\cos\varphi_{j\hat{+}(i+1)}\bigg),\end{aligned}
		%\label{eq37}
	\end{equation*}
	where $a \hat{+} b$ denotes $(a + b)\ mod\, 4$ and $det \boldsymbol{A}$ can be recovered from (\ref{eq31}).
\end{proposition}
\begin{proof}
	Recalling Eqs.\,(\ref{eq19}), (\ref{eq32}) and (\ref{eq35}), after straightforward calculations Eq.\,(\ref{eq36}) is determined.
\end{proof}
Now, for the process $\{(\boldsymbol{X}(t),V(t)), t \geq 0\}$ we are able to formulate the following theorem about the absolutely continuous component concerning  assumption (iv) of  Remark \ref{remark:new}.  For this purpose, we indicate with $f_{D_j^{(n)}}$ the probability density function (p.d.f.) of $D_j^{(n)}$ given in Eq.\,(\ref{eq1}), $1 \leq j\leq 4$, and with $\overline{F}_{D_{j,k+1} \vert D_j^{(k)}} $ the conditional survival function of $D_{j,k+1}$ given $D_{j}^{(k)}$, for $1 \leq j\leq 4$ and $k \in \mathbb{N}_0$.
\begin{theorem}\label{theor2}
	(Absolutely continuous component) For the stochastic process  $\{(\boldsymbol{X}(t),V(t)), t \geq 0\}$ defined in Section \ref{sec2}, under the  initial condition given in Eq.\,(\ref{eq6}) with direction $\vec{v}_1$, we have that, for all $\boldsymbol{x} \in {\rm Int}(\mathcal{T}(t))$ and $t>0$, the absolutely continuous component of the probability law is given by
	\begin{equation}
		\begin{aligned}
			p_{1j}(\boldsymbol{x},t)&=\frac{1}{det\boldsymbol{A}}\sum_{k=0}^{\infty}\bigg\{ \prod_{i=1}^{j-1}f_{D_{i}^{(k+1)}}(\tau_{i})\prod_{i=j+1}^{4}f_{D_{i}^{(k)}}(\tau_{i})\\
			&\times \int_{t-\tau_{j}}^{t} f_{D_{j}^{(k)}}(s-(t-\tau_{j}))\overline{F}_{D_{j,k+1}\vert D_{j}^{(k)}}(t-s \vert s-(t-\tau_{j}))ds\bigg\},
		\end{aligned}
		\label{eq38}
	\end{equation}
	where $\tau_{j}=\tau_{j}(\boldsymbol{x},t)$, for $1 \leq j \leq 4$, is defined in Eq.\,(\ref{eq36}), and $det \boldsymbol{A}$ is given in Eq.(\ref{eq32}).
\end{theorem}
\begin{proof}
	By conditioning on the number of direction switches in $(0,t),$ say $k$, and on the last time $s$ previous $t$ in which the particle changes its velocity from $\vec{v}_4$ to $\vec{v}_1$ in order to restart a new cycle, we reformulate Eq.\,(\ref{eq29}), for  $\boldsymbol{x} \in \mathbb{R}^3$ and $t>0$, as follows:
	\begin{equation}
		\begin{aligned}
			p_{1j}(\boldsymbol{x},t)d\boldsymbol{x}= \sum_{k=0}^{\infty}  \int_{0}^{t} \mathbb{P} \{&T_{4k+j-1}\in ds, X_1(s)+c(t-s)\cos\theta_{j}\sin\varphi_{j}\in dx_1,\\  
			& X_2(s)+c(t-s)\sin\theta_{j}\sin\varphi_{j}\in d x_2, \\
			&X_3(s) + c(t-s)\cos\varphi_{j} \in d x_3,  D_{j,k+1}>t-s\},
		\end{aligned}
		\label{eq39}
	\end{equation}
	for $1 \leq j \leq 4$ and where $T_{4k+j-1}$ is the instant occurring at time $s$ in which the particle changes its direction to velocity $\vec{v}_{j}$. Thus, $\boldsymbol{X}(s)=(X_1(s),X_2(s),X_3(s))$ is the position of the particle when a  change of direction occurs. Therefore, due to Eq.(\ref{eq5}), we have
	\begin{equation*} 
		T_{4k+j-1}= \sum_{i=1}^{j-1}D_{i}^{(k+1)}+  \sum_{i=j+1}^{4}D_{i}^{(k)} + D_{j}^{(k)}.
		%\label{eq40}
	\end{equation*}
	Hence, for $s=T_{4k+j-1}$ and using Eq.\,(\ref{eq1}), we can express the coordinates of the particle position $\boldsymbol{X}(s)$ as follows
	\begin{equation}
		\begin{aligned}
			X_1(s)&=c \Bigg[ \sum_{i=1}^{j-1} D_{i}^{(k+1)}\cos\theta_{i}\sin\varphi_{i}+\sum_{i=j+1}^{4} D_{i}^{(k)}\cos\theta_{i}\sin\varphi_{i} + D_{j}^{(k)}\cos\theta_{j}\sin\varphi_{j} \Bigg], \\
			X_2(s)&=c \Bigg[ \sum_{i=1}^{j-1} D_{i}^{(k+1)}\sin\theta_{i}\sin\varphi_{i}+\sum_{i=j+1}^{4} D_{i}^{(k)}\sin\theta_{i}\sin\varphi_{i} + D_{j}^{(k)}\sin\theta_{j}\sin\varphi_{j} \bigg], \\
			X_3(s)&=c \Bigg[ \sum_{i=1}^{j-1} D_{i}^{(k+1)}\cos\varphi_{i}+\sum_{i=j+1}^{4} D_{i}^{(k)}\cos\varphi_{i}+ D_{j}^{(k)}\cos\varphi_{j}\Bigg]. 
		\end{aligned}
		\label{eq41}
	\end{equation}
	The conditions defined in (\ref{eq25}) and $\boldsymbol{X}(s) \in \mathcal{T}(s)$ implies that $s>t-\tau_{j}$, for $1\leq j\leq 4$. Moreover, using \eqref{eq31} and substituting Eqs.\,(\ref{eq41}) in Eq.\,\eqref{eq39}. we have:
	\begin{equation}
		\begin{aligned}
			p_{1j}(\boldsymbol{x},t)= &\dfrac{1}{\det\boldsymbol{A}}  \sum_{k=0}^{\infty} \bigg\{\int_{t-\tau_{j}}^{t} \Psi_{j,k}\bigg[s,x_1-c(t-s)\sin\varphi_{j}\cos\theta_{j},\\
			& x_2-c(t-s)\sin\varphi_{j}\cos\theta_{j},
			x_3-c(t-s)\cos\varphi_{j}\bigg] \\
			&\times \mathbb{P}\bigg[D_{j,k+1}>t-s \vert  T_{4k+j-1}=s, X_1(T_{4k+j-1})=x_1-x_{1_{\vec{v}_{j}}}(t-s),\\
			& X_2(T_{4k+j-1})=x_2-x_{2_{\vec{v}_{j}}}(t-s),  X_3(T_{4k+j-1})=x_3-x_{3_{\Vec{v}_{j}}}(t-s) \bigg] ds\bigg\},
		\end{aligned}
		\label{eq42}
	\end{equation}
	where $\Psi_{j,k}$ is the joint p.d.f. of $( T_{4k+j-1}, X_1( T_{4k+j-1}),X_2( T_{4k+j-1}),X_3( T_{4k+j-1}))$ with:
	\begin{equation*}
		\begin{aligned}
			T_{4k+j-1}&= \sum_{i=1}^{j-1}D_{i}^{(k+1)}+\sum_{i=j+1}^{4}D_{i}^{(k)} + D_{j}^{(k)}, \\
			X_1 (T_{4k+j-1})&= \sum_{i=1}^{j-1}x_{1_{\vec{v}_{i}}}D_{i}^{(k+1)}+  \sum_{i=j+1}^{4}x_{1_{\vec{v}_{i}}}D_{i}^{(k)} + x_{1_{\vec{v}_{j}}}D_{j}^{(k)},\\
			X_2 (T_{4k+j-1})&= \sum_{i=1}^{j-1}x_{2_{\vec{v}_{i}}}D_{i}^{(k+1)}+  \sum_{i=j+1}^{4}x_2{_{\vec{v}_{i}}}D_{i}^{(k)} + x_{2_{\Vec{v}_{j}}D_{j}}^{(k)},\\
			X_3(T_{4k+j-1})&= \sum_{i=1}^{j-1}x_{3_{\vec{v}_{i}}}D_{i}^{(k+1)}+  \sum_{i=j+1}^{4}x_{3_{\vec{v}_{i}}}D_{i}^{(k)} + x_{3_{\vec{v}_{j}}}D_{j}^{(k)}.
		\end{aligned}
		%\label{eq43}
	\end{equation*}
	According to (\ref{eq30}), for $1\leq j\leq 4$, we have
	\begin{equation*}
		x_{1_{\vec{v}_{i}}}=v_{i}\sin\varphi_{i}\cos\theta_{i}, \ 
		x_{2_{\vec{v}_{i}}}=v_{i}\sin\varphi_{i}\sin\theta_{i},\ 
		x_{3_{\vec{v}_{i}}}=v_{i}\cos\varphi_{i}.
		%\label{eq44}
	\end{equation*}
	Moreover, given the mutual independence of the variables $\{ D_{j,k};  k\in \mathbb{N}\}$, for $1 \leq j\leq 4$, we get
	\begin{equation} 
		\begin{aligned}
			\Psi_{j,k}\bigg[s, x-&c(t-s)\cos\theta_{j}\sin\varphi_{j}, y-c(t-s)\sin\theta_{j}\sin\varphi_{j}, z-c(t-s)\cos\varphi_{j}\bigg]\\
			&=\prod_{i=1}^{j-1}f_{D_{i}^{(k+1)}}(\tau_{i})\prod_{i=j+1}^{4}f_{D_{i}^{(k)}}(t-\tau_{j})f_{D_{j}^{(k)}}[s-(t-\tau_{j})].
		\end{aligned}
		\label{eq45}
	\end{equation}
	Therefore, Eq.\,(\ref{eq38}) is directly obtained replacing Eq.\,(\ref{eq45}) in Eq.\,(\ref{eq42}).
\end{proof}
We point out that Eq.\ (\ref{eq38}) is in analogy with Eq.\ (4.1) of \cite{lachal2006cyclic} in which a particular case of the minimal cyclic random motion in $\mathbb{R}^d$ ($n=d+1$ directions, $n \leq d$) is investigated. Specifically, the integral formulation given in Eq.\,(\ref{eq38}) is based on the application of  Eqs.\,(\ref{eq14}) and (\ref{eq15}) where the assumption of dependent increments, instead of the typical exponential distribution, is taken into account. Differently,  in \cite{orsingher2004cyclic} and \cite{lachal2006minimal} the explicit form of the absolutely continuous part of $\boldsymbol{X}(t)$ is determined but using an approach based on the Bessel functions of higher order.
\par
Note that we analyze the motion in $\mathbb{R}^3$ since the amount of calculations needed is tractable and the obtained results are applicable for practical problems. 
Hence, we can easily conjecture that in $\mathbb{R}^n$ the structure of the absolutely continuous component of the distribution of a cyclic motion with $n+1$ directions (with directions forming a regular or possible irregular tetrahedron) is similar to Eq.\,(\ref{eq38}). We remark that the $n$-dimensional case will be object of a future work.
\par
Clearly, the results given in Theorems \ref{theor1} and \ref{theor2} for the process $\{(\boldsymbol{X}(t),V(t)), t \geq 0\}$ can be extended similarly when the particle starts with velocity $\vec{v}_l$, for $2\leq l \leq 4$. Using this new scheme, a fixed initial direction allows us to obtain tractable expressions for the probability law of the process as illustrated in the following section. 

\section{A special case}\label{sec6}

In this section we analyze a special case of the process $\{(\boldsymbol{X}(t),V(t)), t \geq 0\}$, with initial condition defined by Eq.\,(\ref{eq6}) when $V(0)=\vec{v}_1$ and for fixed directions.
\begin{assumptions} \label{assum1}
	The following conditions hold:
	\begin{enumerate}
		\small
		\item [(i)]The particle's motion is confined in a regular tetrahedron $\mathcal{T}(t)$ characterized by the vertices:
		\begin{equation}
			\begin{aligned}
				A_{1}(t)&=ct\big(1,0,0\big), \quad 
				A_{2}(t)=ct\bigg(-\frac{1}{3}, \frac{2\sqrt{2}}{3},0\bigg), \\
				A_{3}(t)&=ct\bigg(-\frac{1}{3}, -\frac{\sqrt{2}}{3},\sqrt{\frac{2}{3}}\bigg),  \quad
				A_{4}(t)=ct\bigg(-\frac{1}{3}, -\frac{\sqrt{2}}{3},-\sqrt{\frac{2}{3}}\bigg).
				\label{eq46}
			\end{aligned}
		\end{equation}
		\item[(ii)] The vertices in Eq.\,(\ref{eq46}), identified by the directions $\vec{v}_{j}$, for $1 \leq j \leq 4$, satisfy the conditions (i)-(iv) listed in Remark \ref{remark:new}.
		\item[(iii)] The random times $D_{j,n}$, for $n \in \mathbb{N}$ and $1\leq j \leq 4 $, represent the $j$-th duration of the motion within the $n$-th cycle and constitute the intertimes of a GCP with intensity $\lambda_{j} \in \mathbb{R}^{+}$.
	\end{enumerate}
\end{assumptions}
An example of projections onto the state-space $\mathbb{R}^3$ of suitable paths of the process $\{(\boldsymbol{X}(t),V(t)), t \geq 0\}$ under the Assumptions \ref{assum1} is shown in Figure \ref{fig1}.  It is easy to see that  the particle can eventually reach any position of $\mathbb{R}^3$, since $\mathcal{T}(t)  \to \mathbb{R}^3$ as $t \to\infty$.
\begin{figure}[ht]
	\centering
	\includegraphics[scale=0.4]{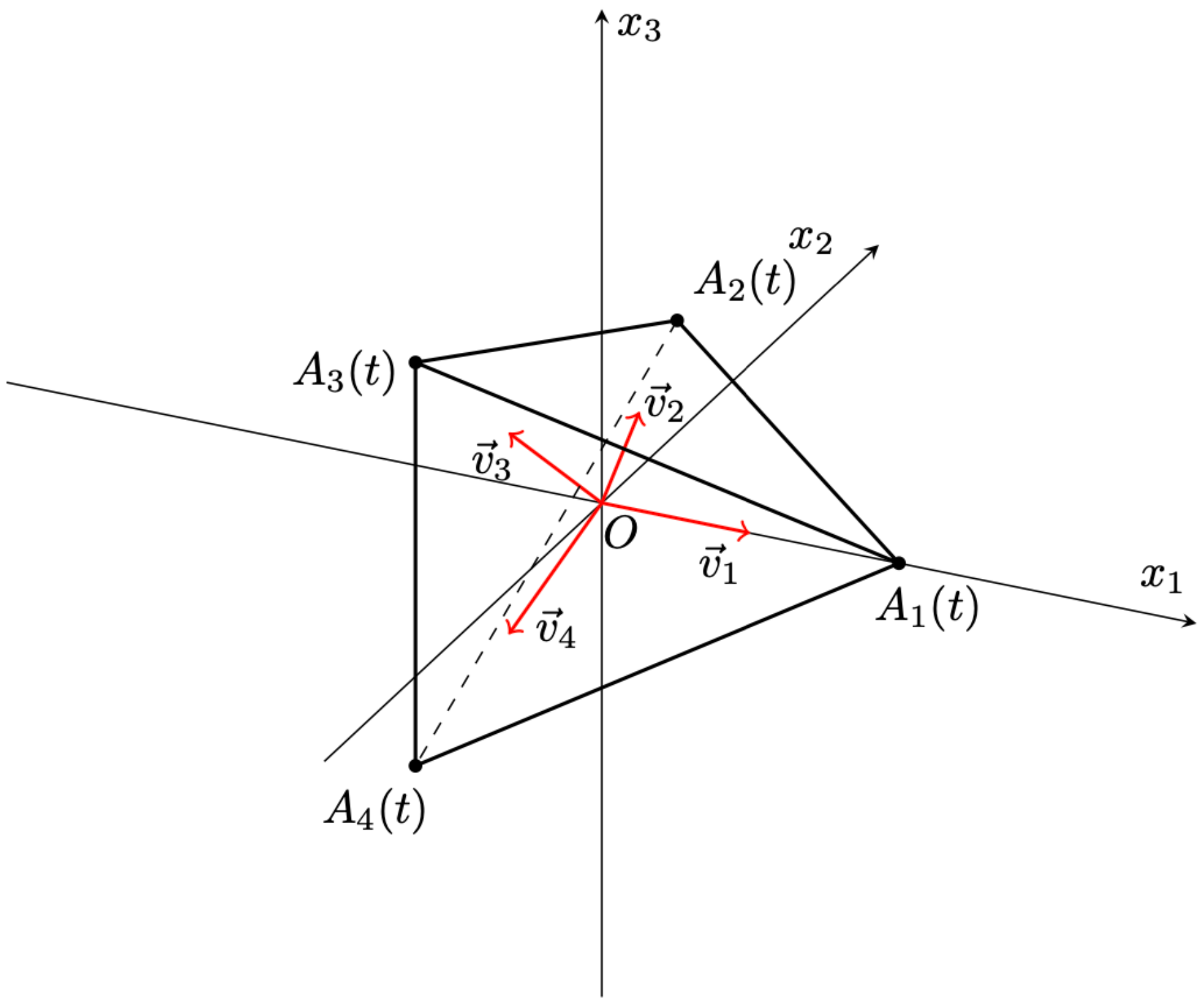}
	\caption{A sample of the set $\mathcal{T}(t)$, with velocities $\vec{v}_j$, for $1 \leq j \leq 4$, and fixed vertices given in (\ref{eq46}).}
	\label{fig1}
\end{figure} 
Moreover, by taking into account the region defined in (\ref{eq25}) and Assumptions \ref{assum1} we have that at every time $t >0$ the set of possible positions $\boldsymbol{x} \in \mathbb{R}^3$ of the moving particle is the tetrahedron
\begin{equation}
	\begin{aligned}
		\mathcal{T}(t)=\bigg\{(x_1, x_2,x_3)\in \mathbb{R}^{3}: \; &-\frac{ct}{3}<x_1<ct; \; -\frac{ct-x_1}{2\sqrt{2}}<x_2<\frac{ct-x_1}{\sqrt{2}}, \\ &\vert x_3 \vert <\frac{\sqrt{6}}{6}(ct-x_1-\sqrt{2}x_2)\bigg\}.
	\end{aligned}
	\label{eq47}
\end{equation}
\begin{remark}
	We observe that the set $\mathcal{T}(t)$ grows as time elapses and, recalling Remark \ref{remark1}, the volume is given by $Vol(\mathcal{T}(t))= \big(\frac{4}{3}\big)^{2}\sqrt{3}(ct)^3$.
\end{remark}
Two sample paths of the region $\mathcal{T}(t)$, defined in (\ref{eq47}), with directions $\vec{v}_j$, $1 \leq j \leq 4$, is illustrated in Figure \ref{fig2}.
\begin{figure}[ht]
	\centering
	\includegraphics[scale=0.4]{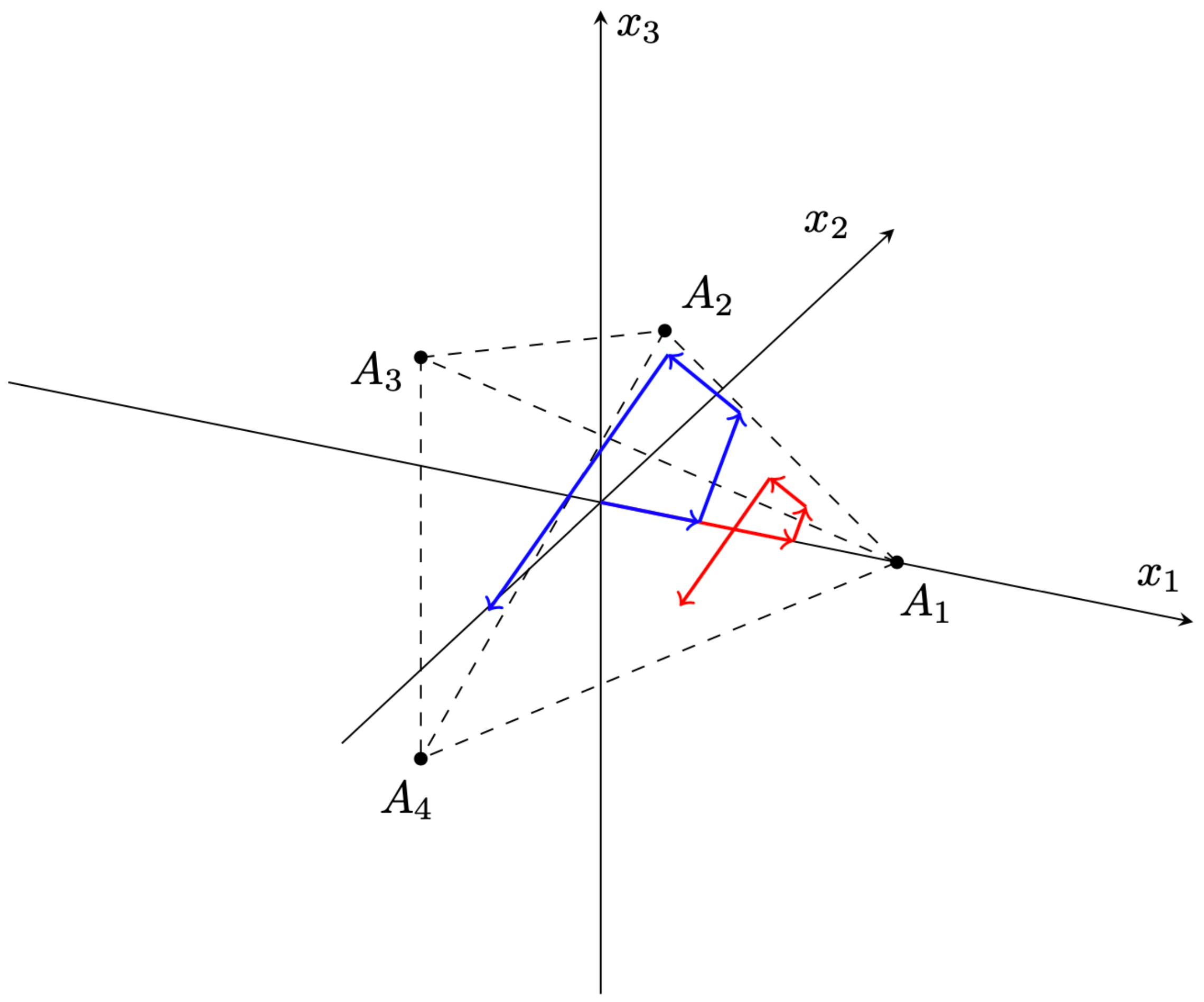}
	\caption{Two possible paths with the first four segments of the cyclic motion defined in Section \ref{sec2} under the conditions of  Remark \ref{prop1} with $c=t=1$. The initial velocity is $\vec{v}_1$.}
	\label{fig2}
\end{figure} 
Under the hypothesis of  Assumptions \ref{assum1}, we are now able to determine the explicit probability laws of the process $\{(\boldsymbol{X}(t),V(t)), t \geq 0\}$ when the random intertimes of the motion along the possible directions $\vec{v}_j$ follow four possible independent GCPs with intensities $\lambda_j$, $1 \leq j \leq 4$.
Specifically, in the following theorem, in order to obtain a tractable form for the third component given in (iii) of Remark \ref{remark:new}, we assume that the sojourn times along the directions $\vec{v}_1$,  $\vec{v}_2$, and $\vec{v}_3$ follow three independent GCPs with identical intensity $\lambda$  (i.e., $\lambda_1=\lambda_2=\lambda_3:=\lambda$).
% %we assume that the parameters of the underlying GCP?s are identical
%
\begin{theorem}\label{theor3}
	(Initial components) Let $\{(\boldsymbol{X}(t),V(t)), t \geq 0\}$ be the stochastic process defined in Section \ref{sec2} under the initial condition defined in Eq.\,(\ref{eq6}) with velocity $\vec{v}_1$. If Assumptions \ref{assum1} hold and the four sequences of intertimes are identically distributed and follow a GCP with intensity $\lambda_j$, $1 \leq j \leq 4$, then for all $t \geq 0$, we have
	\begin{equation}
		\eta_1(t)=\mathbb{P}_1\bigg\{\boldsymbol{X}(t)=(ct, 0, 0), V(t)=\vec{v}_{1}\bigg\}= \dfrac{1}{1+ \lambda_1 t},
		\label{eq48}
	\end{equation}
	\begin{equation}
		\begin{aligned}
			\eta_2(t)&=\mathbb{P}_1\bigg\{ \boldsymbol{X}(t)\in E_{12}, V(t)=\vec{v}_{2} \bigg\}\\
			&=\frac{\lambda_1}{(\lambda_1+\lambda_2+ \lambda_1 \lambda_2 t)^2}\Bigg\{\frac{\lambda_1 t (\lambda_1+\lambda_2+ \lambda_1 \lambda_2 t)}{1+\lambda_1 t}+ \lambda_2 \log\big[(1+\lambda_1 t)(1+\lambda_2 t) \big]\Bigg\}
		\end{aligned}
		\label{eq49}
	\end{equation}
	and 
	\begin{equation}
		\eta_3(t)=\mathbb{P}_1\bigg\{ \boldsymbol{X}(t)\in F_{123}, V(t)=\vec{v}_{3} \bigg\}=R_{\lambda}(t),
		\label{eq50}
	\end{equation}
	where for $\lambda_1=\lambda_2=\lambda_3=\lambda$ one has
	\begin{equation*}
		\begin{aligned}
			R_{\lambda}(t)&=\frac{1}{(2+\lambda t)^2(3+\lambda t)^3}\\
			&\Bigg\{2 \log(1 + \lambda t) \bigg[(3 + \lambda t) \big( \lambda t (3 + 2 \lambda t)\big) +  4 (2 + \lambda t)^2 \log(2 + \lambda t)\bigg]\\
			&+(2 + \lambda t)\bigg[\pi^2 (2 + \lambda t) + \lambda t (3 + \lambda t)^2 - 
			4 (2 +\lambda t) \log\bigg(\frac{1}{2+\lambda t}\bigg)\log \bigg(\frac{1+\lambda t}{2+ \lambda t}\bigg)\bigg]\\
			&+4 (2 +\lambda t)^2\bigg[Li_2\big(-(1+\lambda t)\big)+2 Li_2\bigg(\frac{1+\lambda t}{2+ \lambda t}\bigg)\bigg]\Bigg\}
		\end{aligned}
		%\label{eq52}
	\end{equation*}
	where $Li_2(\cdot)$ is the dilogarithm function.
\end{theorem}
\begin{proof}
	Recalling Theorem \ref{theor1},  under Assumption \ref{assum1}, Eq.\,(\ref{eq48}) is obtained making use of Eq.\,(\ref{eq13}) in the following expression
	\begin{equation*}
		\mathbb{P}_1\{\boldsymbol{X}(t) \in A_1(t), V(t)=\vec{v}_{1}\}= 1-\int_{0}^{t}f_{D_{1,1}}(s){\rm d}s.
		%\label{eq26.1}
	\end{equation*}
	Similarly, Eqs.\,(\ref{eq49}) and (\ref{eq50}), when 
	$V(0)=\vec{v}_2$ and $V(0)=\vec{v}_3$, are given by 
	\begin{equation*}
		\mathbb{P}_1\{\boldsymbol{X}(t) \in E_{12}(t), V(t)=\vec{v}_{2} \}=\int_{0}^{t}f_{D_{1,1}}(s) \int_{t-s}^{\infty}f_{D_{2,1}}(u){\rm d}u\,  {\rm d}s
		%\mathbb{P}\{D_{1,1} < t \leq D_{1,1}+D_{2,1}\} %F_{D_{1,1}+D_{2,1}}(t)-F_{D_{1,1}}(t)
		%\label{eq27.1}
	\end{equation*}
	and
	\begin{equation*}
		\mathbb{P}_1\{\boldsymbol{X}(t) \in F_{123}(t), V(t)=\vec{v}_{3}\}
		= \int_{0}^{t}f_{D_{1,1}+D_{2,1}}(s) \, f_{D_{3,1}}(t-s)  {\rm d}s,
		%\mathbb{P}\{D_{1,1}+D_{2,1} < t \leq D_{1,1}+D_{2,1}+D_{3,1}\},
		%F_{D_{1,1}+D_{2,1}+D_{3,1}}(t)-F_{D_{1,1}+D_{2,1}}(t).
		%\label{eq28.1}
	\end{equation*}
	respectively, after some calculations.
\end{proof}
\begin{figure}[ht]
	\centering
	\includegraphics[scale=0.4]{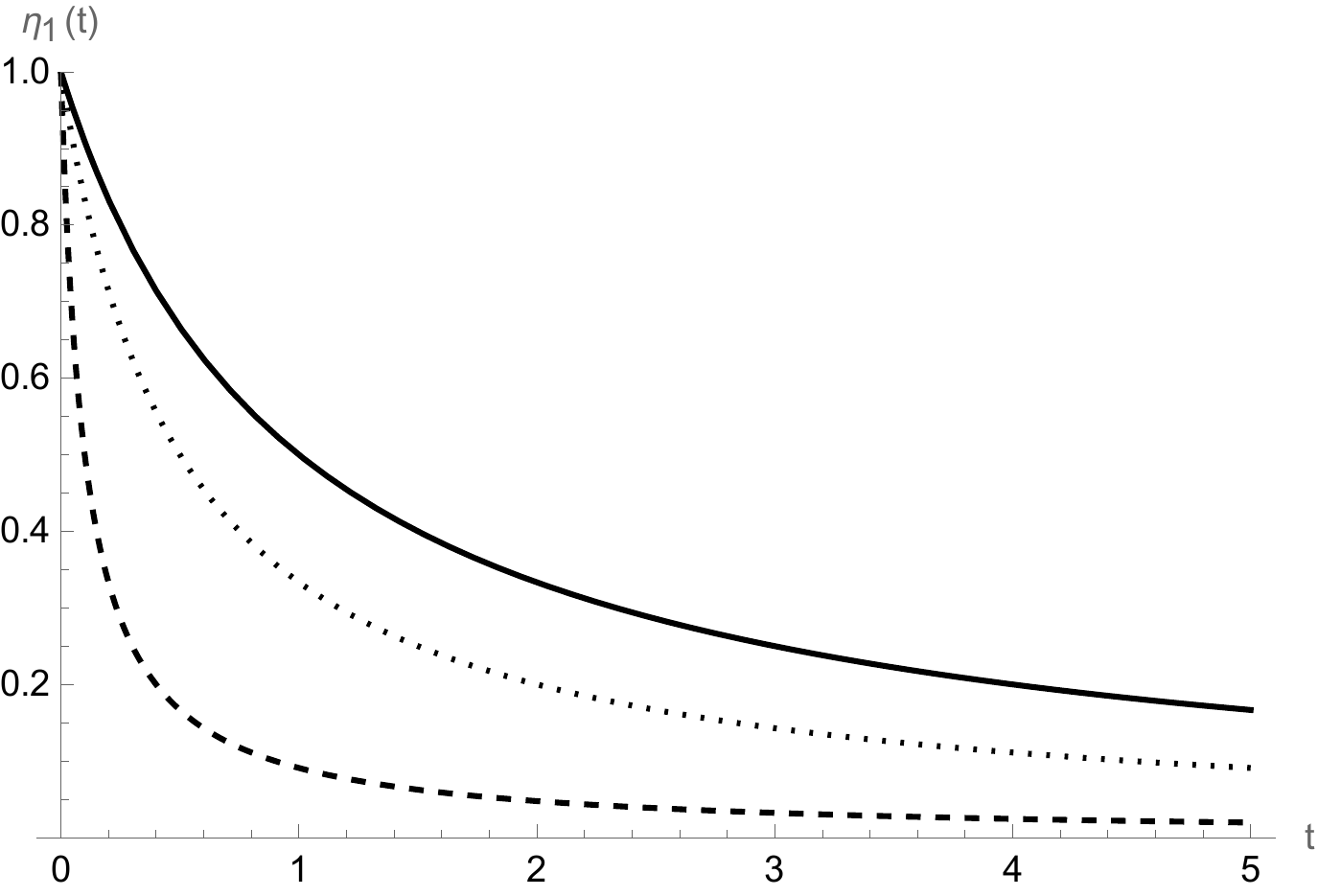}
	\includegraphics[scale=0.4]{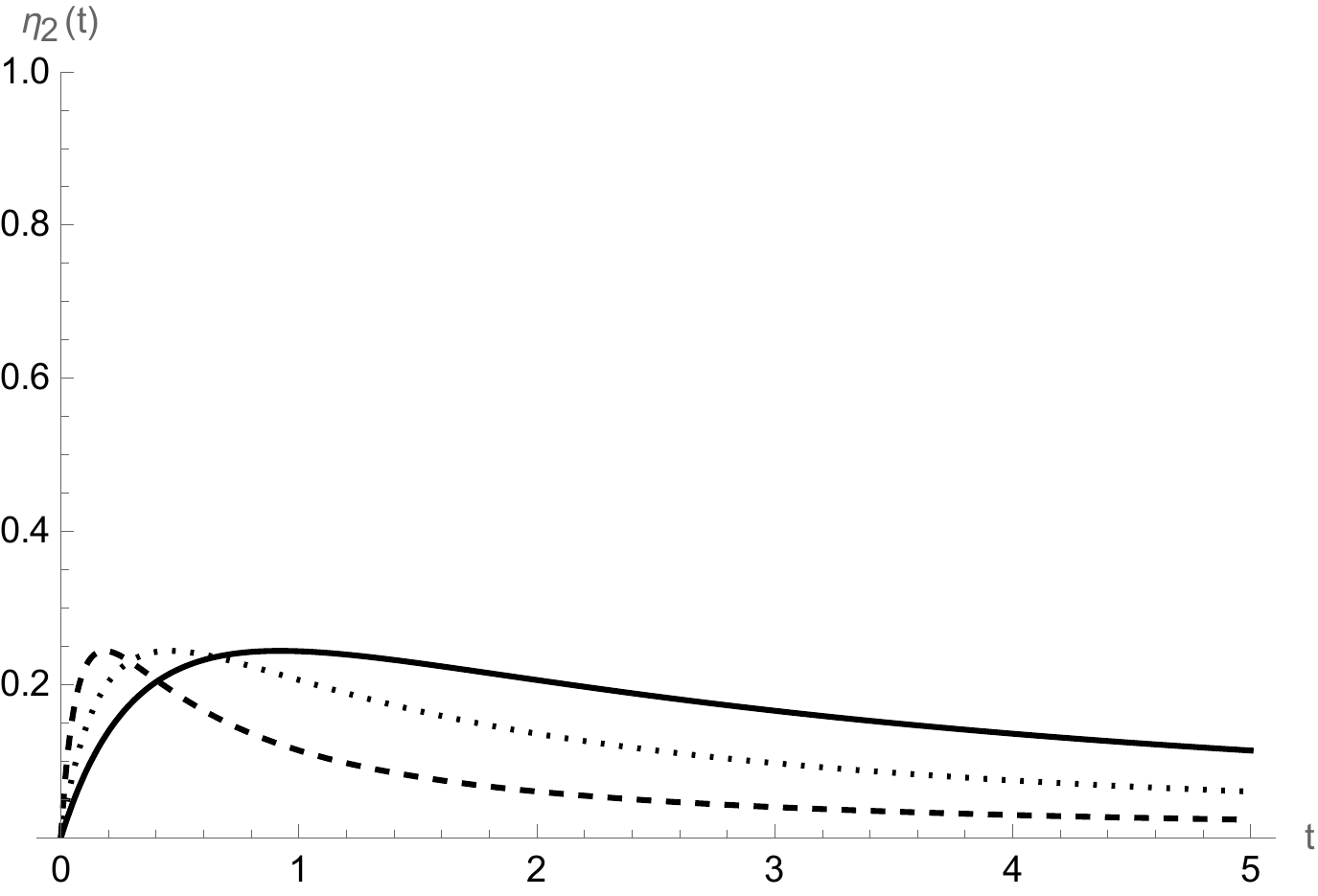}
	\includegraphics[scale=0.4]{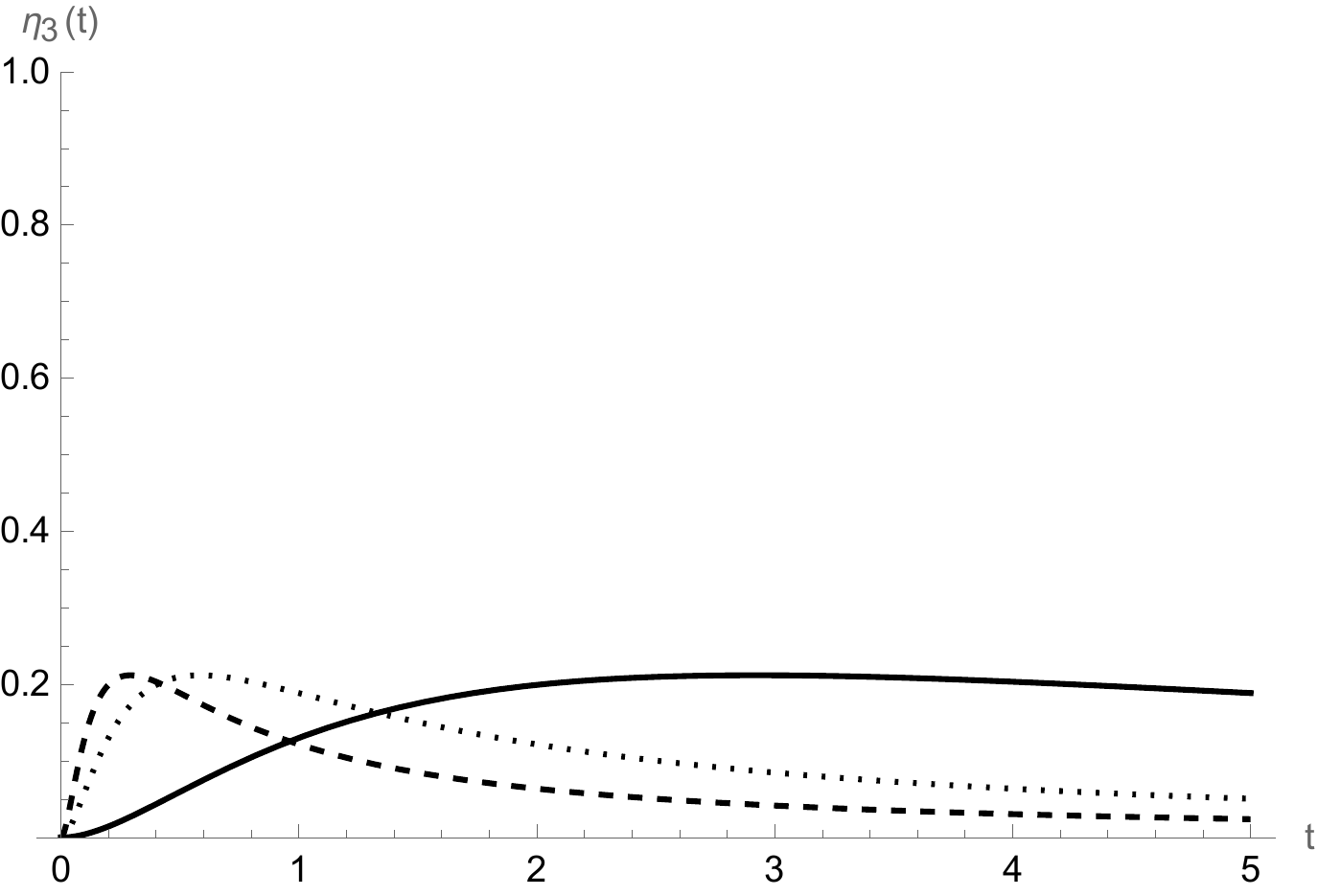}
	\caption{Left to right: plot of  $\eta_1(t)$ with $\lambda_1=$ 1 (solid), 2 (dotted), 10 (dashed);  plot of  $\eta_2(t)$ with $\lambda_1=1$ and  $\lambda_2=2$ (solid line), $\lambda_1=2$ and  $\lambda_2=4$ (dotted line), $\lambda_1=5$ and  $\lambda_2=10$ (dashed line); plot of  $\eta_3(t)$ with $\lambda=$ 1 (solid), 5 (dotted), 10 (dashed).}
	\label{fig3}
\end{figure} 
%
% It is worth mentioning that the Eqs.\,(\ref{eq48}) and (\ref{eq49}) are in agreement with the analogous results given in Eqs.\,(104) and (105) of \cite{di2023some}.
%
In Figure \ref{fig3} are shown some plots of the probabilities given in Eqs.\,(\ref{eq48}), (\ref{eq49}) and (\ref{eq50}) for different choices of the intensities $\lambda_j$, $1 \leq j \leq 3$, respectively.
%\par
%For the absolutely continuous component we formulate the theorem below assuming that the switching times follow four independent GCPs with intensity $\lambda_j$, for $1 \leq j \leq 4$.
%
\begin{theorem}\label{theor4}
	(Absolutely continuous components) Let $\{(\boldsymbol{X}(t),V(t)), t \geq 0\}$  be the stochastic process defined in Section \ref{sec2} under the initial condition defined in Eq.\,(\ref{eq6}) with velocity $\vec{v}_1$. For all $\boldsymbol{x} \in {\rm Int}(\mathcal{T}(t))$, with $\mathcal{T}(t)$ given in (\ref{eq47}), we have
	\begin{equation} 
		\begin{aligned}
			p_{11}(\boldsymbol{x},t)&=\lambda_{1}\lambda_{2}\lambda_{3}\lambda_{4}\tau_{1}\\
			&\times \frac{\big[1+A(\boldsymbol{\tau})+B(\boldsymbol{\tau})+C(\boldsymbol{\tau})\big]^{2}+6D(\boldsymbol{\tau})\big[1+A(\boldsymbol{\tau})+B(\boldsymbol{\tau})+C(\boldsymbol{\tau})+D(\boldsymbol{\tau})\big]}{det\boldsymbol{A}\big[1+A(\boldsymbol{\tau})+B(\boldsymbol{\tau})+C(\boldsymbol{\tau})\big]^{4}},\\
			p_{12}(\boldsymbol{x},t)&=2\lambda_{1}^{2}\lambda_{2}\lambda_{3}\lambda_{4}\tau_{1}\tau_{2}(1+\lambda_{2}\tau_{2})(1+\lambda_{3}\tau_{3})(1+\lambda_{4}\tau_{4})\\
			&\times \frac{\big[1+A(\boldsymbol{\tau})+B(\boldsymbol{\tau})+C(\boldsymbol{\tau})+3D(\boldsymbol{\tau})\big]}{det\mathcal{A}\big[1+A(\boldsymbol{\tau})+B(\boldsymbol{\tau})+C(\boldsymbol{\tau})\big]^{4}},\\
			p_{13}(\boldsymbol{x},t)&=2\lambda_{1}^{2}\lambda_{2}^{2}\lambda_{3}\lambda_{4}\tau_{1}\tau_{2}\tau_{3}(1+\lambda_{3}\tau_{3})(1+\lambda_{4}\tau_{4})\\
			&\times \frac{\big[1+A(\boldsymbol{\tau})+B(\boldsymbol{\tau})+C(\boldsymbol{\tau})\big]+3D(\boldsymbol{\tau})}{det\boldsymbol{A}\big[1+A(\boldsymbol{\tau})+B(\boldsymbol{\tau})+C(\boldsymbol{\tau})\big]^{4}},\\
			p_{14}(\boldsymbol{x},t)&=\lambda_{1}\lambda_{2}\lambda_{3}(1+\lambda_{4}\tau_{4})\\
			&\times \frac{\big[1+A(\boldsymbol{\tau})+B(\boldsymbol{\tau})+C(\boldsymbol{\tau})\big]^{2}+6D(\boldsymbol{\tau})[1+A(\boldsymbol{\tau})+B(\boldsymbol{\tau})+C(\boldsymbol{\tau})+D(\boldsymbol{\tau})]}{det \boldsymbol{A}\big[1+A(\boldsymbol{\tau})+B(\boldsymbol{\tau})+C(\boldsymbol{\tau})\big]^{4}},
		\end{aligned}
		\label{eq53}
	\end{equation}
	with
	\begin{equation*} 
		\begin{aligned}
			det \boldsymbol{A}=\frac{16\sqrt{3}}{9} c^{3}&, \quad A(\boldsymbol{\tau})= \sum_{i=1}^{4}\lambda_{i}\tau_{i}, \quad B(\boldsymbol{\tau})= \sum_{\substack{i,j=1\\i<j}}^{4}\lambda_{i}\lambda_{j}\tau_{i}\tau_{j}, \\
			C(\boldsymbol{\tau})&= \sum_{\substack{i,j,k=1\\i<j<k}}^{4}\lambda_{i}\lambda_{j}\lambda_{k}\tau_{i}\tau_{j}\tau_{k}, \quad D(\boldsymbol{\tau})=\prod_{i=1}^{4} \lambda_{i}\tau_{i},
		\end{aligned}
		%\label{eq54}
	\end{equation*}
	and where the terms $\tau_{j}=\tau_{j}(\boldsymbol{x},t)$, for $1\leq j \leq 4$, are given by
	\begin{equation} 
		\begin{aligned}
			&\tau_{1}=\frac{ct+3x_1}{4c}, \quad  \tau_{2}=\frac{ct-x_1+2\sqrt{2}x_2}{4c}, \\ \tau_{3}=&\frac{ct-x_1-\sqrt{2}x_2+\sqrt{6}x_3}{4c}, \quad  \tau_{4}=\frac{ct-x_1-\sqrt{2}x_2-\sqrt{6}x_3}{4c}.
			\label{eq55}
		\end{aligned}
	\end{equation}
\end{theorem}
\begin{proof}
	Eqs.\,(\ref{eq53}) are obtained in a closed form as an immediate consequence of Theorem \ref{theor2} after straightforward calculations.
\end{proof}
We observe that the values of $\tau_{j}$ in (\ref{eq55}) are obtained directly from Eq.\,(\ref{eq36}) in Proposition \ref{prop2}. Moreover, due to Eq.\,(\ref{eq9}), under the assumptions of Theorem \ref{theor4} the p.d.f.\ $p_1(\boldsymbol{x},t)$ can be immediately obtained from Eqs.\,(\ref{eq53}).
\begin{figure}[ht]
	\centering
	\includegraphics[scale=0.45]{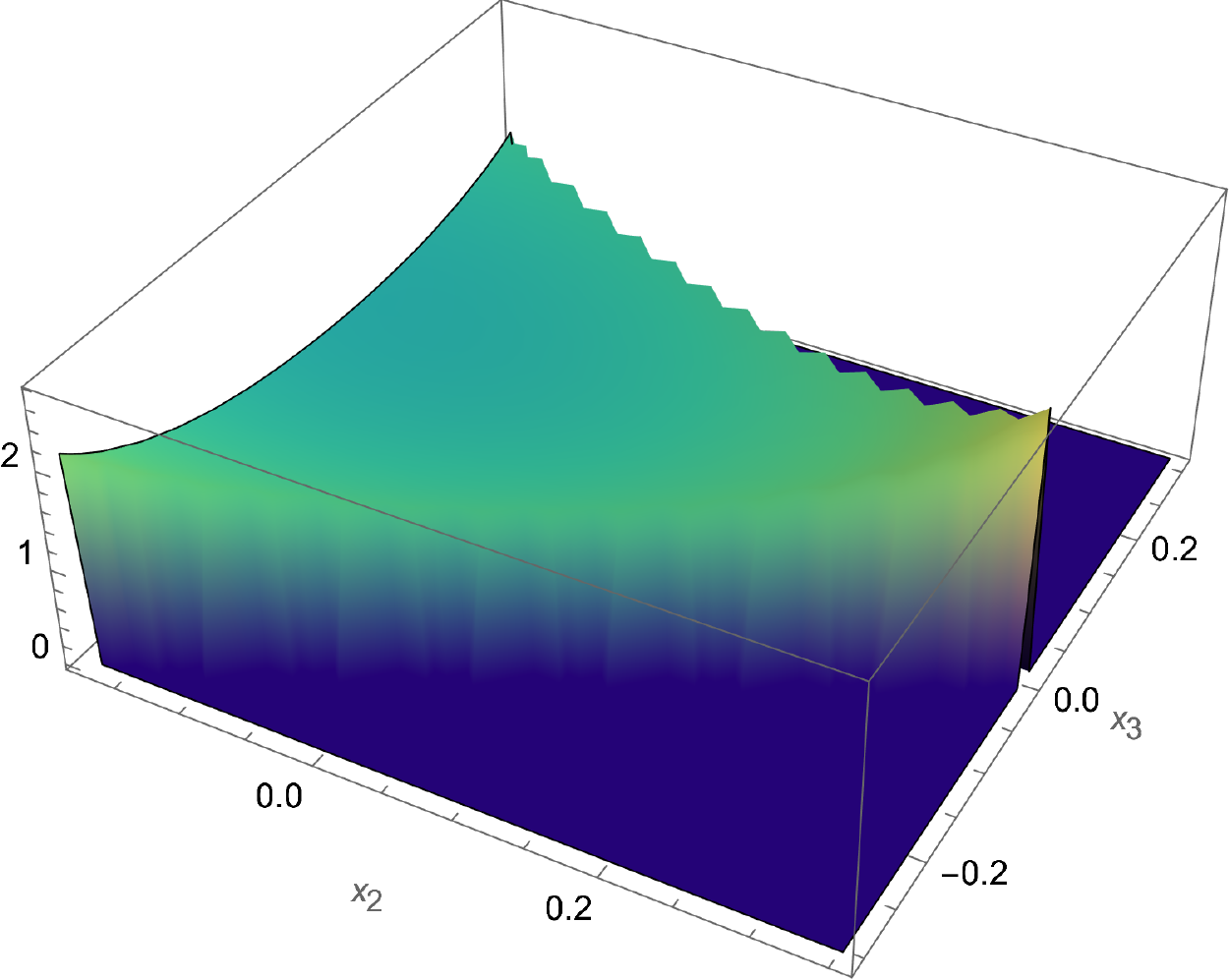}
	\includegraphics[scale=0.45]{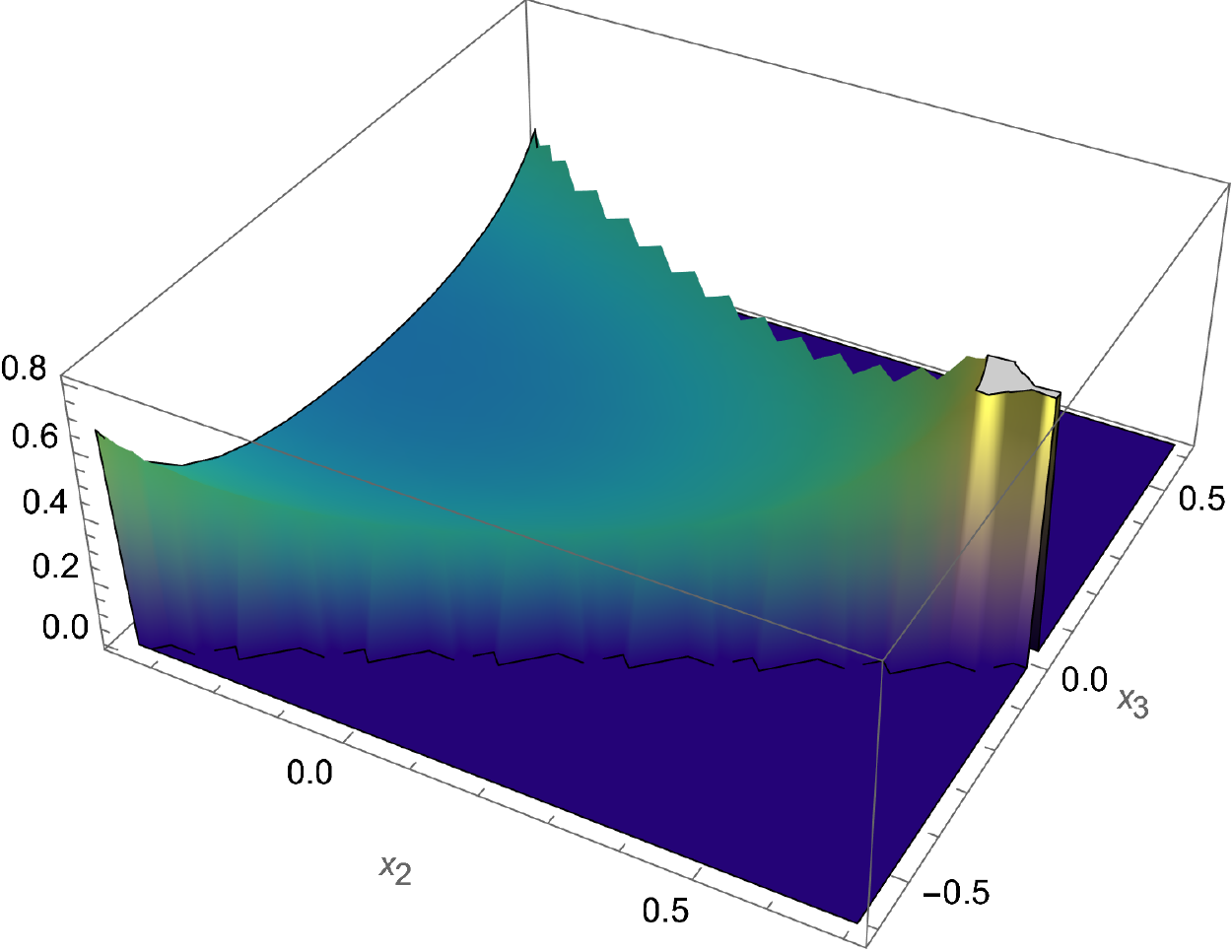}
	\caption{Plot of $p_1(\boldsymbol{x},t)$ for $c=t=1$ and $x_1=1/2$ on the left-side and $c=1$, $t=2$ and $x_1=1$ on the rigth-side.}
	\label{fig4}
\end{figure} 
\par 
As example, some plots of  $p_{1}(\boldsymbol{x},t)$  are illustrated in Figure \ref{fig4} for fixed $c>0$ and $t>0$  and $x_1 \in (-ct/3, ct)$. Note that the domain for $(x_2,x_3)$ has a triangular form. Moreover, from Eq.\,(\ref{eq47}) and Theorem \ref{theor4} it is clear that the dependence on $c$ and $t$ is expressed through the product $ct$.
\par
Hereafter, we study the asymptotic behaviour of the p.d.f. of the particle location $\boldsymbol{X}(t)$ defined in (\ref{eq9}) when the intensities $\lambda_j$ tend to infinity, for $1 \leq j\leq 4$. In particular, substituing Eqs.\,(\ref{eq53}) in Eq.\,(\ref{eq9}) we obtain the following corollary.
\begin{corollary}\label{corol1}
	Under the assumptions of Theorem \ref{theor4}, for $t >0$ and $\boldsymbol{x} \in {\rm Int}(\mathcal{T}(t))$ one has
	\begin{equation*}
		\lim_{\substack{\forall j,\, \lambda_{j} \to + \infty \\
				\forall i,\, \lambda_{1}/\lambda_{i}\to 1}}
		p_{1}(\boldsymbol{x},t)= \xi(\boldsymbol{x},t),
		%\label{eq56}
	\end{equation*}
	where $\xi(\boldsymbol{x},t)$ is the following p.d.f.\
	\begin{equation}
		\xi(\boldsymbol{x},t)=\dfrac{6 t (\tau_{1}\tau_{2}\tau_{3}\tau_{4})^{2}}{det \boldsymbol{A}[(\tau_{1}\tau_{2}\tau_{3})^{4}+(\tau_{1}\tau_{2}\tau_{4})^{4}+(\tau_{1}\tau_{3}\tau_{4})^{4}+(\tau_{2}\tau_{3}\tau_{4})^{4}]},
		\label{eq57}
	\end{equation}
	with $\tau_{j}$ expressed in (\ref{eq36}).
\end{corollary}
Some instances of  $\xi(\boldsymbol{x},t)$ given in Eq.\,(\ref{eq57}) are plotted in Figure \ref{fig5} for fixed $c>0$ and $t>0$, with $x_1 \in (-ct/3, ct)$.
\begin{figure}[ht]
	\centering
	\includegraphics[scale=0.45]{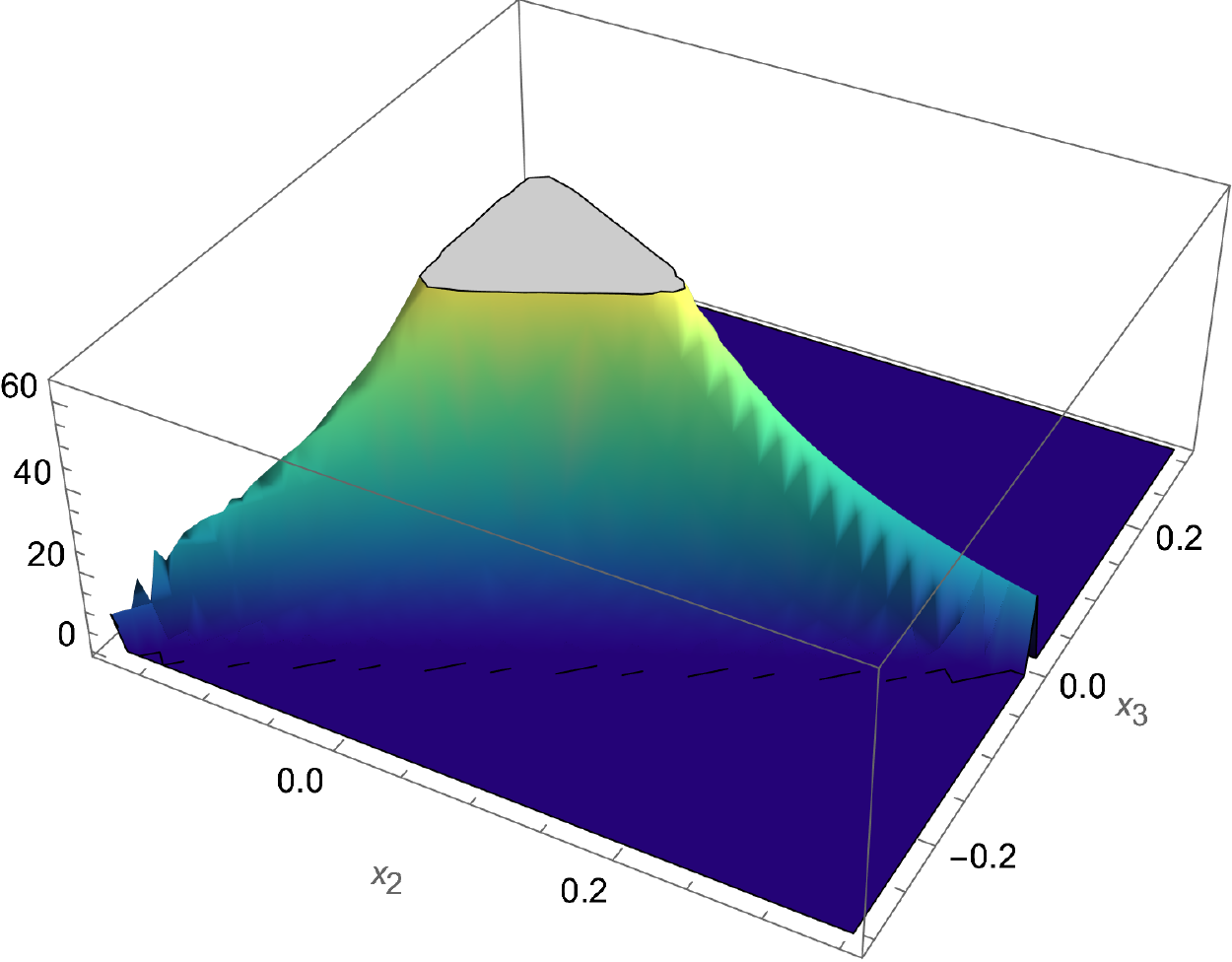}
	\includegraphics[scale=0.45]{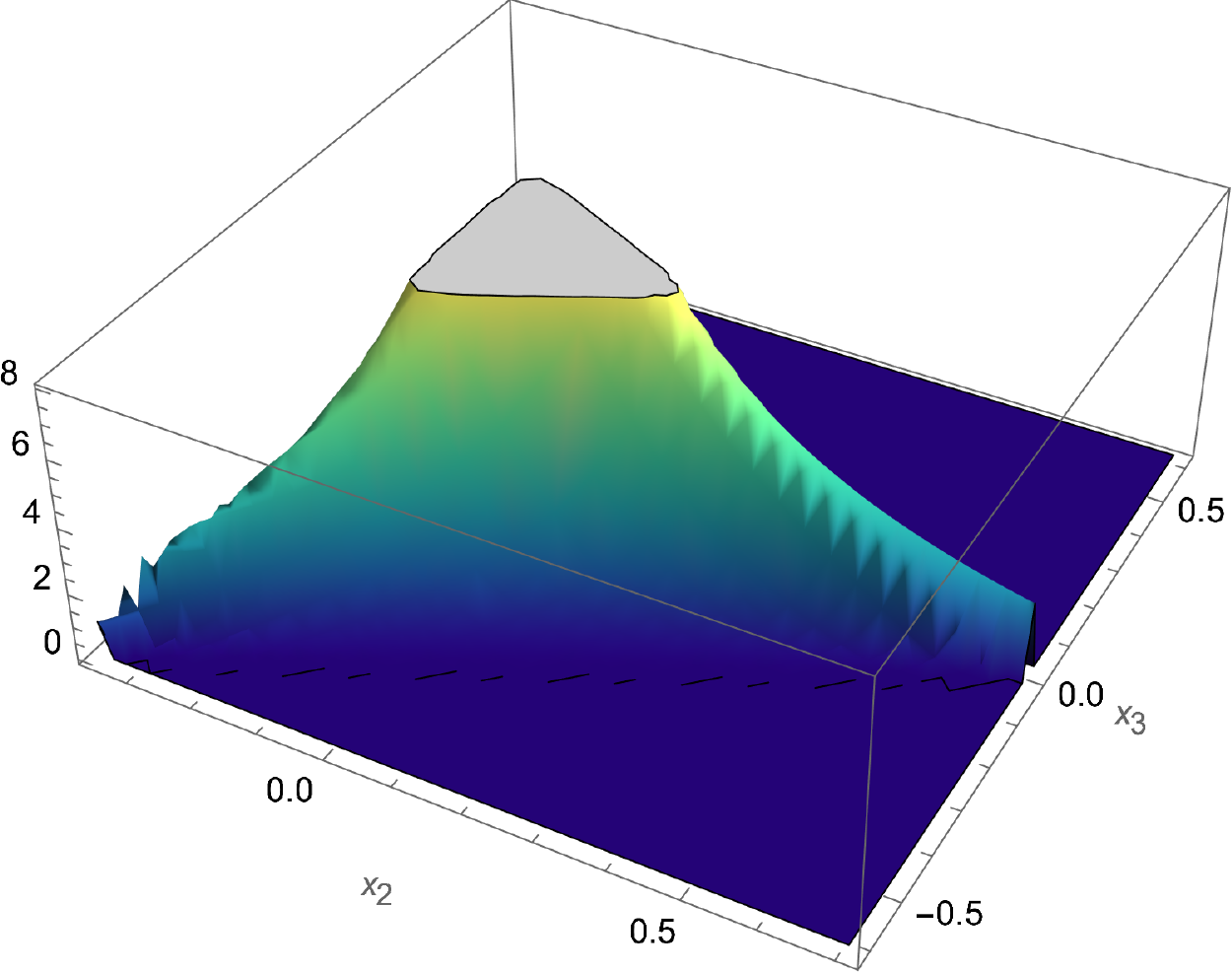}
	\caption{Plot of $\xi(\boldsymbol{x},t)$ for $c=t=1$ and $x_1=1/2$ on the left-side and $c=1$, $t=2$ and $x_1=1$ on the rigth-side.}
	\label{fig5}
\end{figure} 
\begin{remark}
	Under the assumptions of Theorem \ref{theor4}, we note that, for $t >0$ and $1 \leq j \leq 4$, the limit of $p_1(\boldsymbol{x},t)$ for $(x_1,x_2,x_3) \to (ct \cos \theta_j \sin\varphi_{j}, ct\sin\theta_{j} \sin\varphi_{j}, ct \cos \varphi_{j})$ can be computed in a closed form that we omit due to its complexity.
\end{remark}
Similarly to the classical telegraph process driven by the Poisson process (see, e.g., Lemma 2 of \cite{orsingher1990probability}), it is not hard to see that the process $\boldsymbol{X}(t)$ does not admit a stationary state.
In the next corollary making use of Eq.\,(\ref{eq9}) we analyse the asymptotic behaviour of density $p_1(\boldsymbol{x},t)$ when the time $t$ tends to $+\infty$. 
% In particular,  as $t$ approches infinity, we first study the terms $\tau_j$ in (\ref{eq55}), $1 \leq j \leq 4$, and the quantities in (\ref{eq54}), then we investigate the performance of each $p_{1,j}(\boldsymbol{x},t)$ in (\ref{eq53}), for $1 \leq j \leq 4$.
%
\begin{corollary}\label{corol2}
	Under the assumptions of Theorem \ref{theor4}, for $t >0$ and $\boldsymbol{x} \in {\rm Int}(\mathcal{T}(t))$ one has
	\begin{equation*}
		\lim_{t \to + \infty} p_{1}(\boldsymbol{x},t)= 0.
		%\label{eqtime}
	\end{equation*}
\end{corollary}
\begin{proof}
	Recalling Eq.\,(\ref{eq9}), the stated result follows by noting that each density $p_{1,j}(\boldsymbol{x},t)$ in (\ref{eq53}), for $1 \leq j \leq 4$, behaves as $t^{-3}$ as $t$ tends to infinity. 
	%To examine the behaviour of $p_1(\boldsymbol{x},t)$  as $t$ approches infinity, we observe that each densities $p_{1,j}(\boldsymbol{x},t)$ in (\ref{eq53}), for $1 \leq j \leq 4$, tends to be an indeterminate form $\frac{\infty}{\infty}$. Specifically, the values in (\ref{eq55}) have order $t$, while the terms $A(\boldsymbol{\tau})$,  $B(\boldsymbol{\tau})$, $C(\boldsymbol{\tau})$, and $D(\boldsymbol{\tau})$ in (\ref{eq54}) possess order $t$, $t^2$, $t^3$ and $t^4$, respectively. 
	%
	% Hence,  comparing the sizes of the two infinities (numerator and denominator) of each density in (\ref{eq53}) we have that the highest order element is below of the fraction, i.e., $t^9/t^{12}$. Therefore, recalling Eq.\,(\ref{eq9}) the limit in Eq.\,(\ref{eqtime}) is immediately obtained.
\end{proof}
%
%The right-hand side of (\ref{eqtime}) shows that the distribution of the process $\boldsymbol{X}(t)$, when the initial velocity is $\vec{v}_1$, is asymptotically uniformly distributed inside the diffusion region, i.e., ${\rm Int}(\mathcal{T}(t))$.

The results expressed in this section for $V(0)=\vec{v}_1$ can be extended to the cases $V(0)=\vec{v}_j$, for $2 \leq j\leq 4$, by using a similar strategy.

\section{Concluding remarks}\label{sec7}

In this paper we analyzed a finite random motion in $\mathbb{R}^3$ where the sojourn times along each direction form four independent GCPs, and where the possible directions alternate cyclically. This work has been inspired by  \cite{di2023some} with the aim to de?ne similar processes in higher dimensions with possibly variable velocities.
Potential applications in biomathematics, engineering, financial and actuarial sciences allow to investigate possible future developments  also oriented  to the study of the first-passage-time problem.

Here, in order to illustrate the basic issues of this problem, we introduce the (upward) first-passage time for the first component of  the process $\{(\boldsymbol{X}(t),V(t)), t \geq 0\}$ through a constant barrier, say $\beta > 0$, conditional on  $C_1$ (cf.\ Eq.\,(\ref{eq6})) given by
\begin{equation}
	\tau_{\beta} = \inf\{t \geq 0: X_1(t) \geq \beta\}, \qquad {\bf X}(0)={\bf 0}, \quad V(0)=\vec v_1, \quad \vert{\vec{v_1}}\vert= c,
	\label{eq58}
\end{equation}
for $c>0$. 
The probability distribution of (\ref{eq58}) can be expressed in terms of the sub-density functions:
\begin{equation*}
	g_\beta(t,k):= \frac{\mathbb{P}_1\{\tau_{\beta} \in dt, N(t)=k\}}{dt}, \quad k\in \mathbb{N}.
	%\label{eq59}
\end{equation*}
In particular, by the law of total probability, we express the conditional distribution of $\tau_\beta$ in the form
\begin{equation}
	\mathbb{P}_1\{\tau_{\beta} \in dt\}= \mathbb{P}\{D_{1,1}>t\}\delta_{\frac{\beta}{c}}(dt)+\sum_{k=1}^{+\infty} \mathbb{P}_1\{\tau_{\beta} \in dt, N(t)=k\},
	\label{eq60}
\end{equation}
where $\delta_{\frac{\beta}{c}}$ is the Dirac delta measure at $\frac{\beta}{c}$, and $N(t)$ is the alternating counting process introduced in Eq.\ (\ref{eq2}).  The first term on the right-hand-side of (\ref{eq60}) corresponds to the motion without any direction switching up to time $t$. The series on the right-hand side of (\ref{eq60}) represents the absolutely continuous component of the first-passage-time distribution, which arises when at least one direction reversal occurs. We also recall that $D_{1,1}$, defined in Eq.\,(\ref{eq1}), is the random duration of the first time interval during which the particle proceeds with velocity $\Vec{v}_1$. 
In order to describe the first-passage-time problem, we consider as threshold the plane $x_1=\beta$ and  project the vectors $\vec{v}_j$, for $1 \leq j \leq 4$,  onto the $x_1$-axes according to the following relation
\begin{equation*}
	\vec{v}_{j_{x_1}}=proj_{\vec{w}_1}(\vec{v}_j)=\frac{\vec{w}_1 \cdot \vec{v}_j}{\vert \vert \vec{w}_1\vert \vert ^2}\cdot \vec{w}_1,
	\qquad 1 \leq j \leq 4.
\end{equation*}
The latter is the projection of the vector $\vec{v}_j$ on     $\vec{w}_1$, which is the versor along the $x_1$-axes.  An illustration of the problem is shown in Figure \ref{fig6}. 
\begin{figure}[ht]
	\centering
	\includegraphics[scale=0.4]{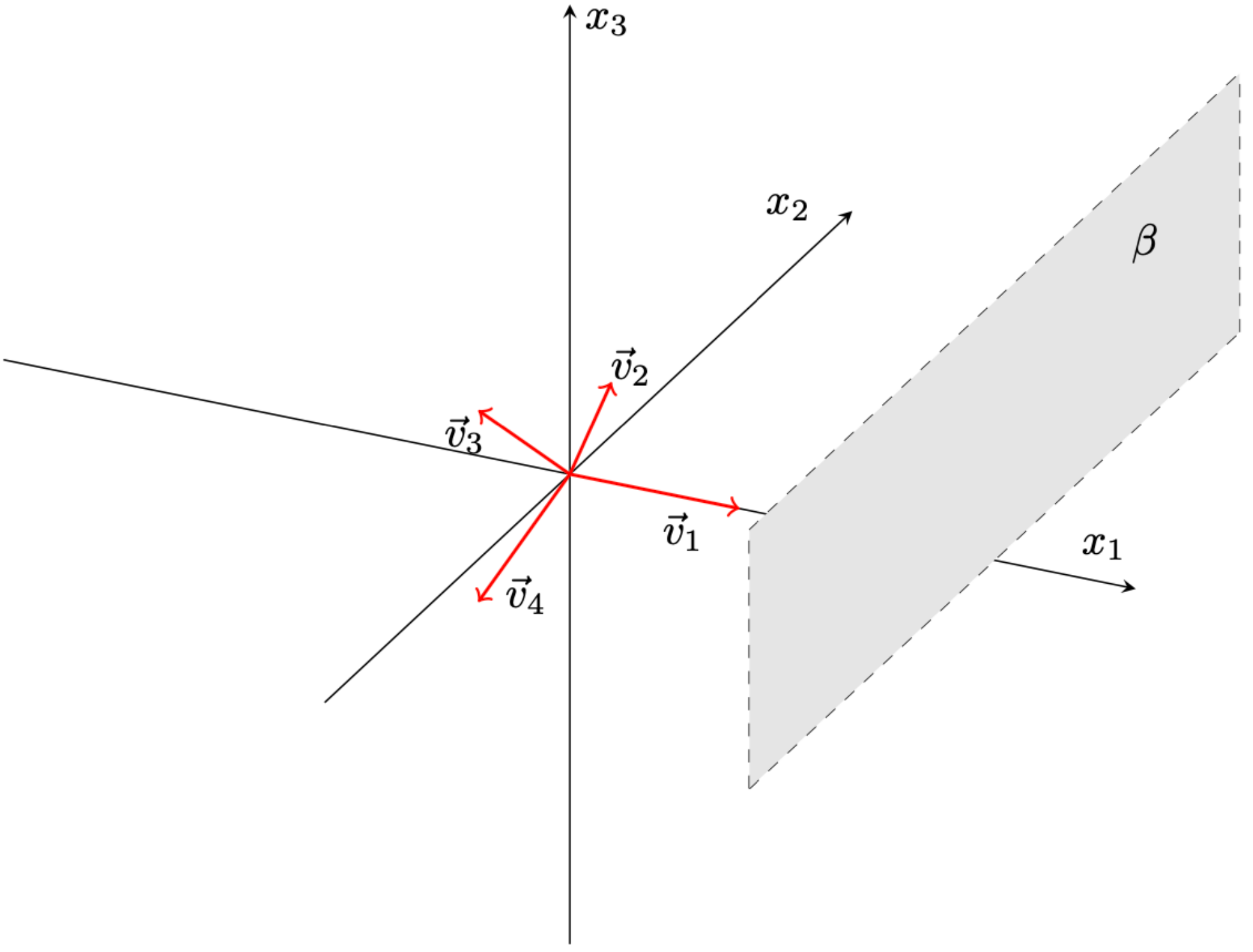}
	\caption{Illustration the constant barrier $\beta > 0$ for the first-passage time problem introduced in (\ref{eq58}).}
	\label{fig6}
\end{figure} 
Clearly, under Assumption \ref{assum1} of Section \ref{sec6}, we have:
\begin{equation*}
	\begin{aligned}
		\vec{v}_{1_{x_1}}&=c(1,0,0), \quad \vec{v}_{2_{x_1}}=c\big(-\frac{1}{3},0,0\big),\\
		\vec{v}_{3_{x_1}}&=c\big(-\frac{1}{3},0,0\big), \quad \vec{v}_{4_{x_1}}=c\big(-\frac{1}{3},0,0\big).
	\end{aligned}
	%\label{eq61}
\end{equation*}
\par
Due to the complexity of the problem, we discuss only the first cycle of particle motion concerning $D_{j,1}$ ($1 \leq j \leq 4$). 
We point out that the particle may reach the threshold $x_1=\beta$ only when it runs with velocity $\vec{v}_{1_{x_1}}$, since the other directions $\vec{v}_{2_{x_1}}$,  $\vec{v}_{3_{x_1}}$ and $\vec{v}_{4_{x_1}}$ refer to motion in the opposite direction. Moreover, with reference to the first term in the right-hand-side of (\ref{eq60}), one has 
\begin{equation*}
	\begin{aligned}
		\mathbb{P}(D_{1,1}>t) &=\mathbb{P}\bigg\{X_1(\tau_{\beta})=\beta, V(\tau_{\beta})=\vec{v}_{1}\bigg\}\\ %&=\mathbb{P}\bigg\{c\tau_{\beta}=\beta, N(\tau_{\beta})=0\bigg\}\\
		&=\mathbb{P}\bigg\{\tau_{\beta}=\frac{\beta}{c}, N(\tau_{\beta})=0\bigg\}\\
		&=\overline{F}_{D_{1,1}}\left(\frac{\beta}{c}\right)=
		\frac{1}{1+\lambda_1{\beta}/{c}}.
		%\label{eq62}
	\end{aligned}
\end{equation*}
Clearly, during the random periods $D_{2,1}, D_{3,1}, D_{4,1}$, when the particle moves with velocity $\vec{v}_{2_{x_1}}= \vec{v}_{3_{x_1}}=\vec{v}_{4_{x_1}}=-\frac{c}{3}$, it cannot reach the threshold $\beta$ since it moves in the opposite direction.
Therefore, the position occupied by the particle at the end of the first period of motion $D_1^{(1)}$ is given by
\begin{equation*}
	X_1(t)=cD_{1,1}-\frac{c}{3}\big(D_{2,1}+D_{3,1}+D_{4,1}\big).
	%\label{eq63}
\end{equation*}
It is worth mentioning that the determination of an explicit form for the terms in the series of  (\ref{eq60}) is in general very difficult even when $k$ is small and when the intensities $\lambda_i$ are equal. Hence, in view of possible future developments, in a forthcoming investigation we aim to apply computational methods to determine the related probabilities.
\par
In conclusion, we stress that the analysis of finite-velocity random motions in multidimensional domains deserves interest in various applied fields. In particular, the motion of a particle in a three-dimensional space as studied in this paper provides possible applications also in chemistry, since the tetrahedral geometry characterizes the shape of many molecules. Indeed, the motion of a particle is isotropic, i.e., the direction of its movement is uniformly distributed on the unit sphere in $\mathbb{R}^3$.
% For instance, the motion of a free electron or a positively charged particle within a tetrahedral crystalline lattice occurs through semiconductor doping techniques \cite{roncali2007one}. 
Thus, according to this interpretation, the Eq.\,(\ref{eq8}) can be employed to represent the wave function of the electron. 
\par
Further interesting real applications of finite-velocity random motions for modelling  random occurrences of events in time and space are related, for instance, to biology (for  the random motions of microorganisms), to geology (for  alternating trends in volcanic areas), and to physics (for the vorticity motion in two or more dimensions, see \cite{orsingher2008random}). Moreover, finite-velocity random motions in $\mathbb{R}^3$ are also useful to describe the movements of particles in gases, see, for instance, \cite{reimberg2013cmb}, where the authors introduce an application to the study of photon propagation in the Cosmis Microwave Background (CMB) radiation.  At least, using a similar approach, relevant real applications in higher dimensions concerning cyclic random motions under a GCP can be explored in future works even if the computation complexity in the resolution of the probability law is a very hard task.
Hence, in our view the finite-velocity random motion discussed here is an extension of the telegraph process to the space $\mathbb{R}^3$ and it is probably one of the possible ways for which the explicit distribution of the position $\boldsymbol{X}(t)$, under the assumption introduced in Section \ref{sec3}, can be determined.

\bmhead{Acknowledgments}
The authors are members of the group GNCS of INdAM (Istituto Nazionale di Alta Matematica). 
%Acknowledgments are not compulsory. Where included they should be brief. Grant or contribution numbers may be acknowledged.
%
%Please refer to Journal-level guidance for any specific requirements.
%
\section*{Declarations}

\noindent
{\bf Competing Interest} The authors have no other relevant financial or non-financial interests to disclose.

\bigskip
\noindent
{\bf Funding} This work is partially supported by INdAM-GNCS (project ``Modelli di shock basati sul processo di conteggio geometrico e applicazioni alla sopravvivenza'', CUP-E55F22000270001). This work is partially supported by  PRIN 2022 PNRR, project P2022XSF5H ``Stochastic Models in Biomathematics and Applications'' .

\bigskip
\noindent
{\bf Author Contribution statement}
AI and GV contributed equally to this work.

\bigskip
\noindent
{\bf Data Availability statement}
Not applicable.

\end{document}